\documentclass[11pt, psamsfonts]{amsart}

\usepackage{amssymb,amsfonts}
\usepackage[all,arc]{xy}
\usepackage{amsaddr}
\usepackage{enumerate}
\usepackage{mathrsfs}
\usepackage{thmtools}
\usepackage{thm-restate}
\usepackage{datetime}
\usepackage{hyperref}
\usepackage{graphicx}
\usepackage{cleveref}
\usepackage{comment}
\usepackage{dirtytalk}
\usepackage{float}
\usepackage{longtable}
\usepackage{xcolor}
\usepackage{pdfpages}
\definecolor{amethyst}{rgb}{0.6, 0.4, 0.8}
\definecolor{darkspringgreen}{rgb}{0.09, 0.45, 0.27}

\calclayout

\setboolean{@twoside}{false}

\newtheorem{thm}{Theorem}[section]
\newtheorem{cor}[thm]{Corollary}
\newtheorem{prop}[thm]{Proposition}
\newtheorem{lem}[thm]{Lemma}

\theoremstyle{definition}
\newtheorem{defn}[thm]{Definition}

\newtheorem{exmp}[thm]{Example}

\theoremstyle{remark}
\newtheorem{rem}[thm]{Remark}

\makeatletter
\let\c@equation\c@thm
\makeatother
\numberwithin{equation}{section}

\makeatletter
\newcommand*\bigcdot{\mathpalette\bigcdot@{.5}}
\newcommand*\bigcdot@[2]{\mathbin{\vcenter{\hbox{\scalebox{#2}{$\m@th#1\bullet$}}}}}
\makeatother

\makeatletter
\def\subsection{\@startsection{subsection}{3}%
  \z@{.5\linespacing\@plus.7\linespacing}{.1\linespacing}%
  {\bfseries}}
\makeatother

\newcommand{\R}{\mathbb{R}}
\newcommand{\Hy}{\mathbb{H}}

\renewcommand{\phi}{\varphi}
\DeclareMathOperator{\conv}{conv}
\DeclareMathOperator{\rank}{rank}
\DeclareMathOperator{\cut} { \setminus}

\usepackage{etoolbox}
\makeatletter
\patchcmd{\@maketitle}
  {\ifx\@empty\@dedicatory}
  {\ifx\@empty\@date \else {\vskip3ex \centering\footnotesize\@date\par\vskip1ex}\fi
   \ifx\@empty\@dedicatory}
  {}{}
\patchcmd{\@adminfootnotes}
  {\ifx\@empty\@date\else \@footnotetext{\@setdate}\fi}
  {}{}{}
\makeatother

\title[Compact Hyperbolic Coxeter Polytopes with Few Facets]{Near Classification of Compact Hyperbolic Coxeter $d$-Polytopes with $d+4$ Facets and Related Dimension Bounds}
\author{Amanda Burcroff}
\address{Department of Mathematics,
Harvard University}
\email{aburcroff@math.harvard.edu}

\begin{document}
\maketitle

\begin{abstract}
    We complete the classification of compact hyperbolic Coxeter $d$-polytopes with $d+4$ facets for $d=4$ and $5$. By previous work of Felikson and Tumarkin, the only remaining dimension where new polytopes may arise is $d=6$. We derive a new method for generating the combinatorial type of these polytopes via the classification of point set order types. In dimensions $4$ and $5$, there are $348$ and $51$ polytopes, respectively, yielding many new examples for further study.  
    
    We furthermore provide new upper bounds on the dimension $d$ of compact hyperbolic Coxeter polytopes with $d+k$ facets for $k \leq 10$.  It was shown by Vinberg in $1985$ that for any $k$, we have $d \leq 29$, and no better bounds have previously been published for $k \geq 5$.  As a consequence of our bounds, we prove that a compact hyperbolic Coxeter $29$-polytope has at least $40$ facets.
\end{abstract}

\section{Introduction}
Let $\Hy^d$ be the $d$-dimensional real hyperbolic space.  A hyperbolic Coxeter polytope is a domain in $\Hy^d$ bounded by a collection of geodesic hyperplanes, such that each intersecting pair of hyperplanes meets at dihedral angle $\frac{\pi}{m}$ for some integer $m \geq 2$.  Hyperbolic Coxeter polytopes are precisely the fundamental domains of discrete hyperbolic reflection groups.  These polytopes also have relevance to the construction of orbifolds and manifolds, in particular some of minimal volume \cite{Kel}.  

In this paper, we classify the compact hyperbolic Coxeter $d$-polytopes with $d+4$ facets for $d = 4$ and $5$, as well as improve some bounds on the dimension of compact Coxeter polytopes with few facets.  While Euclidean and spherical Coxeter polytopes were classified by Coxeter in 1934 \cite{Cox}, no complete classification is known in the hyperbolic case.  Henceforth, all polytopes are assumed to be hyperbolic unless otherwise specified.  

The partial classification of compact Coxeter polytopes has been obtained by restricting either the dimension, combinatorial type, or number of facets.  A dynamic summary of this progress is maintained by Anna Felikson on her webpage %spacing
\\\url{www.maths.dur.ac.uk/users/anna.felikson/Polytopes/polytopes.html}.  

This work is based on the author's master's thesis \cite{Bur} at Durham University, supervised by Pavel Tumarkin.  The thesis is publicly available at \url{http://etheses.dur.ac.uk/14202/}.  

\subsection{\texorpdfstring{Classification of compact Coxeter $d$-polytopes with $d+4$ facets for $d \neq 6$}{Classification of compact Coxeter d-polytopes with d+4 facets for d not equal to 6}}
We first focus on restricting the number of facets with respect to the dimension.  Compact Coxeter simplices, i.e., $d$-polytopes with $d+1$ facets, were classified by Lann{\'e}r in 1950 \cite{Lan}.  These arise only in dimensions $2$, $3$, and $4$.  The compact Coxeter $d$-polytopes with $d+2$ facets are classified in \cite{Kap} and \cite{Ess2}; these arise in dimensions $3$ through $5$ and must be a simplicial prism except in dimension $4$.  Esselmann \cite{Ess} showed in 1994 that a compact Coxeter $d$-polytope with $d+3$ facets must satisfy $d \leq 8$, and that there is a unique polytope of dimension $8$ (first constructed by Bugaenko \cite{Bug}).  In 2007, Tumarkin \cite{Tum} completed the classification of compact Coxeter $d$-polytopes with $d+3$ facets, which arise in dimensions $2$ through $6$ and $8$.  

The first portion of this paper is dedicated to furthering the classification of compact Coxeter $d$-polytopes with $d+4$ facets.   In 2008, Felikson and Tumarkin \cite{FT} showed that such polytopes arise only in dimension at most $7$, and furthermore that there is a unique compact Coxeter $7$-polytope with $11$ facets (originally constructed by Bugaenko \cite{Bug}).  We complete this classification in dimensions $4$ and $5$, with $348$ polytopes in dimension $4$ and $51$ of dimension $5$.  This includes the first known Coxeter polytope in dimension $> 3$ with an angle of less than $\frac{\pi}{10}$ and the first known Coxeter polytope in dimension $> 3$ with an angle of $\frac{\pi}{7}$, along with many new essential polytopes.  A polytope is \emph{essential} if it is minimal with respect to the operations of taking the fundamental domain of a finite index reflection subgroup of the corresponding reflection group, or gluing two Coxeter polytopes along congruent facets (see \cite{FT3} for further details).  The present work combined with that of Felikson and Tumarkin yields a classification in all dimensions except $6$, where the only known polytope was constructed by Bugaenko \cite{Bug}.  We show that a compact Coxeter $6$-polytope with $10$ facets must contain a missing face of size $3$ or $4$.

In order to obtain this classification, we develop a new method for restricting the possible combinatorial types of these polytopes.  The Gale diagram of a $d$-dimensional polytope with $n$ facets is an $n - d - 1$ arrangement of points in Euclidean space that encodes the combinatorial type of the polytope.  In studying $d$-polytopes with $d+4$ facets, this means their combinatorial types can be studied in terms of a $3$-dimensional point arrangement.  Moreover, each Gale diagram can be transformed into an affine Gale diagram, an arrangement of positive and negative points, to further reduce the dimension by $1$ \cite{Vin}.  We show that the (affine) Gale diagrams of compact $d$-dimensional hyperbolic polytopes with $d+4$ facets can be generated by bipartitioning the points in all point set order types with $d+4$ points (see Theorem \ref{thm: affine generation} for further details).  The point set order types of sizes up to $10$ have been enumerated and made available through the Point Set Order Type Database \cite{AK}.  Using this database, we produce a reasonably short list of possible combinatorial types for the polytopes of interest in dimensions $4$ and $5$.  In dimension $6$, the same methods can be applied, but the number of point set order types makes the process rather computationally demanding.

Having greatly restricted the possible combinatorial types in dimensions $4$ and $5$, we then determine whether each combinatorial type can be realised as one or more polytopes.  This involves enumerating weighted graphs with restrictions on certain subgraphs and the spectral properties of their adjacency matrices.  Though the search space is infinite, combinatorial and linear algebraic techniques (see, e.g., \cite{Vin}) have previously been successful in reducing this to a computational problem.  In particular, Tumarkin handled the analogous task for polytopes with $d + 3$ facets by inspecting local determinants, face structures, and gluings of Lann{\'e}r diagrams \cite{Tum}.  These methods are only partially effective for the polytopes with $d + 4$ facets, due to the greater complexity of the polytopes and less restrictive missing face structures.  We then utilise the computer algebra system {\tt Mathematica} to check a finite number of cases in order to list all polytopes of a given combinatorial type, a technique which was recently used in classifying the compact Coxeter cubes \cite{JT}.

Ma and Zheng independently and via different methods classified the compact hyperbolic Coxeter $4$-polytopes with $8$ facets \cite{MZ} and $5$-polytopes with $9$ facts \cite{MZ2}. Their work became publicly available within a few months following the release of the author's master's thesis on this topic (available at \url{http://etheses.dur.ac.uk/14202/}). The author is very grateful to Ma and Zheng for their communication about this classification, as it helped to correct several minor errors.  Due to the sheer volume of data handling required by both our methods and those of Ma-Zheng, there were a few errors in the polytope lists initially announced by both groups.  In the case of the author's master's thesis, coding errors led to the omission of ten $4$-polytopes of type $G_2$ (all obtainable from one such polytope by gluing prisms) and one $5$-polytope of type $H_6$.  There were also three polytopes of type $G_3$ that were double-counted due to isomorphic Coxeter diagrams being listed.  Ma, Zheng, and the current author now agree on the published lists.  The existence of two independent methods for obtaining these polytopes may lend some confidence to the accuracy of this rather delicate classification.

\subsection{Bounding the dimension of polytopes with few facets}
In Sections \ref{sec: 3-free} and \ref{sec: dim bound}, we shift our focus to bounding the dimension of certain compact Coxeter polytopes.  It was shown in 1984 by Vinberg \cite{Vin2} that compact Coxeter polytopes do not arise in dimensions higher than $29$.  Vinberg proceeded by constructing certain weightings on the edges of the polytopes, and utilised a result of Nikulin \cite{Nik} on the average number of vertices along a $2$-dimensional face.  In Section \ref{sec: 3-free}, we show that a slight modification of Vinberg's argument yields a stronger bound for $3$-free polytopes, that is, polytopes having missing faces only of order $2$.  In particular, we show that compact Coxeter $3$-free polytopes do not arise in dimension higher than $13$. \footnote{Since the release of this work, our upper bound of $13$ has been improved using similar methods (which, in turn, were inspired by those of Vinberg \cite{Vin2}) to $12$ by Alexandrov \cite{Ale}.}  

In Section \ref{sec: dim bound}, we improve the bounds on the dimension of compact Coxeter $d$-polytopes with $d+k$ facets for $5 \leq k \leq 10$.  In order to obtain an initial bound, we examine certain faces which must themselves be compact Coxeter polytopes, similar to the methods used by Felikson and Tumarkin \cite{FT} to bound the dimension when $k = 4$.  These ideas combined with the results in Section \ref{sec: 3-free} yield the bounds in Theorem \ref{thm: partial dim bound}.  The rest of the section is devoted to improving these bounds in certain cases, frequently referring to the classification of polytopes with fewer facets.  One corollary of our improved bounds is that any compact Coxeter polytopes of dimension $29$, i.e., the threshold of Vinberg's bound, must have at least $29 + 11 = 40$ facets.

\section{Convex Polytopes and Their Combinatorial Types}

 \subsection{Convex polytope preliminaries}
 A convex polytope in $\Hy^d$ can be defined as $\bigcap_{i \in I} H_i^-$, where $I$ is an index set and $H_i^-$ is a half-space containing $P$. Throughout this paper we will consider only polytopes of finite volume, or equivalently, those which can be obtained as the convex hull of a finite point set.  We furthermore assume that the set of bounding hyperplanes is chosen minimally to define $P$.    The intersection of a convex polytope with a bounding hyperplane is called a \emph{facet} of the polytope.  We identify the index set $I$ with the set $\{0,1,\dots,d+k-1\}$, where $d+k$ is the number of facets of $P$.  If the convex polytope $P$ has a vertex at infinity, then $P$ is said to be \emph{non-compact}, otherwise we say $P$ is \emph{compact}. For the purposes of this section, we do not assume $P$ is compact, though we only consider compact polytopes in the remainder of this paper.  If each vertex is formed by the intersection of precisely $d$ half spaces, then $P$ is said to be \emph{simple}. An equivalent condition is that every $(d-j)$-face is contained in precisely $j$ facets. 
 
Fix an enumeration of the facets of $P$ as $f_0$, $f_1$, \dots, $f_{d+k-1}$.  Then each face $P \cap f_{i_0}\cap \cdots \cap f_{i_s}$ for $0 \leq i_0 < \dots < i_s < d+k$ is denoted by the string $i_0 \dots i_s$.  A \emph{missing face} of $P$ is a list of facets whose intersection is empty, but such that the intersection of every proper subset of these facets is non-empty.  That is, a missing face is an intersection of facets $i_0\dots i_s$ that is empty, but such that $i_0\dots i_{m-1}i_{m+1}\dots i_s$ is non-empty for each $0 \leq m \leq s$.  We refer to the set of missing faces of $P$ as the \emph{missing face list}.  Note that every missing face list is an antichain by inclusion, i.e., no missing face is a subset of another.  Two polytopes are said to be \emph{isomorphic}, or of the same \emph{combinatorial type}, provided there exists a bijective correspondence between their faces, such that two faces of the first polytope meet if and only if the corresponding faces of the second meet.  In particular, two polytopes have the same combinatorial type if and only if they have the same set of missing faces, up to relabelling of the facets.

\subsection{Gale diagrams and affine Gale diagrams}\label{subsec: gale diagrams}
An important technique in classifying convex polytopes is representing a polytope by a ``diagram'' from which one can read off the face structure.  Throughout this paper, we frequently reference \emph{Gale diagrams}, introduced in a $1956$ paper by David Gale \cite{Gal}.  Note that while Gale diagrams are often defined using the vertices of a polytope, we look at the dual construction defined on the facets.  We will consider only simple polytopes in this section, which include all compact Coxeter polytopes (see Remark \ref{rem: simple Coxeter}).

Given a simple $d$-polytope $P$ with $d + k$ facets, a Gale diagram of $P$ consists of a set of $d + k$ points on $S^{k-2} \subset \R^{k-1}$ corresponding to the facets of $P$.  These points can be obtained by applying a \emph{Gale transform} to the affine dependences of the normal vectors to the facets (see \cite{GS} for details of the construction).  These points encode the face structure of $P$ in the following way: a set of facets $\{f_i : i \in I\}$ of $P$ has a non-trivial intersection if and only if the points in the Gale diagram corresponding to $\{f_j : j \in [d+k] \cut I\}$ have a convex hull containing the origin.  Two Gale diagrams are \emph{isomorphic} if their corresponding polytopes are combinatorially equivalent.

A set of $d+k$ points on $S^{k-2}$ corresponds to a Gale diagram of a simple convex $d$-polytope with $d+k$ facets if and only if every half space bounded by a hyperplane through the origin contains at least two of the points.  One can see that the latter condition is necessary by taking the convex hull of the points corresponding to $\{f_j: j \neq i\}$ for any $i \in [d+k]$; this convex hull should contain the origin, as any facet is itself a non-empty face.

In the case $k = 4$, Gale diagrams consist of point configurations on $S^2 \subseteq \R^3$.  These seem rather difficult to classify, and thus a crucial step in our analysis is passing to the affine Gale diagrams.  These affine variants encode the same information as Gale diagrams, but using partitioned sets of points in $\R^{k-2}$. Thus, the configurations are reduced to points on the Euclidean plane when $k = 4$, which can be classified using point set order types (further details in Section \ref{sec: comb types d+4}).

An \emph{affine Gale diagram} of a $d$-polytope with $d+k$ facets consists of two (not necessarily disjoint) point sets, called ``positive'' and ``negative'', in $\R^{k-2}$ containing $d+k$ points in total (counted with multiplicity).  An affine Gale diagram is obtained from a Gale diagram $G$ by taking a hyperplane $H$ through the origin not containing any points of $G$, and projecting the points orthogonally onto $H$, with the projections of points from the open half space $H^+$ being labelled ``positive'' and those from $H^-$ labelled ``negative''.  The face structure of $P$ is determined in the following way: a set of facets $\{f_i : i \in I\}$ of $P$ has a non-trivial intersection if and only if the convex hull of the positive points in $\{f_j : j \notin I\}$ non-trivially intersects the convex hull of the negative points in $\{f_j : j \notin I\}$.

A configuration of positive and negative points then corresponds to an affine Gale diagram of a polytope if and only if, for every hyperplane through the origin $H$ in $\R^{k-2}$, the total number of positive points contained in the open half space $H^+$ and negative points contained in $H^-$ is at least $2$.  This classification follows directly from the analogous half space condition for Gale diagrams.

\section{\texorpdfstring{Combinatorial Types of Simple $d$-Polytopes with $d+4$ Facets}{Combinatorial Types of Simple d-Polytopes with d + 4 Facets}}\label{sec: comb types d+4}
\subsection{\texorpdfstring{Affine Gale diagrams of simple $d$-polytopes with $d+4$ facets}{Affine Gale diagrams of compact polytopes with d + 4 facets}}\label{sec: affine Gale d+4}
In order to determine the compact hyperbolic $d$-polytopes with $d+4$ facets in dimension $4$ and $5$, our methods involve first limiting their possible combinatorial types.  This was accomplished for compact hyperbolic $d$-polytopes with $d+3$ facets by Tumarkin \cite{Tum}, using \emph{standard Gale diagrams} described by Esselmann\cite{Ess}.  However, Tumarkin's methods do not seem immediately generalisable to Gale diagrams of higher dimension. We develop a new method for generating affine Gale diagrams using a classification of point set order types by Aichholzer, Aurenhammer, and Krasser \cite{AAK}.  These are reduced to a representative list of Gale diagrams, from which one can easily determine the set of missing faces. Later sections are devoted to determining which of these combinatorial types are realisable. 

We now describe how to obtain the potential combinatorial types of compact hyperbolic $d$-polytopes with $d + 4$ facets, listed in Appendix \ref{app: Gale diagrams}. First, we show that every combinatorial type of such a polytope has a corresponding affine Gale diagram satisfying certain properties (see Subsection \ref{subsec: gale diagrams} for details on affine Gale diagrams).  In fact, the conditions we consider hold for the more general case of simple hyperbolic $d$-polytopes with $d+4$ facets.  

\begin{rem}\label{rem: disj pos neg}
Let $P$ be a simple $d$-polytope with $d + k$ facets for $k \geq 2$.  Observe that the positive and negative points of an affine Gale diagram for $P$ are necessarily disjoint.  If not, let $v \in \R^2$ be a point which is both positive and negative.  The facets not corresponding to this point then intersect non-trivially in a face of codimension $d + k - 2 \geq d$, which is not possible.
\end{rem}

A set of points in $\R^d$ is said to be in \emph{general position} if no $m$ points lie in a subspace of dimension $m -2$ for $m = 2,\dots,d+1$.  In particular, when $d = 2$ this is equivalent to requiring that no three points are collinear.

\begin{prop}\label{prop: gen pos}
For $d \geq 4$, every simple $d$-polytope with $d + 4$ facets admits an affine Gale diagram where all points are in general position. 
\end{prop}
\begin{proof}
Let $A = A_+ \cup A_-$ be an affine Gale diagram for $P$, where $A_+, A_-\subseteq \R^2$ denote the set of positive points and the set of negative points, respectively.  Our aim will be to slightly perturb the points of $A$ to ensure they are in general position while preserving the associated combinatorial type.  By Remark \ref{rem: disj pos neg}, $A_+$ and $A_-$ are necessarily disjoint.  

Since every vertex of $P$ is obtained as the intersection of exactly $d$ facets, then no set of fewer than $4$ positive and negative vertices can have intersecting convex hulls.  In particular, no positive vertex lies on the line segment between two negative vertices, and no negative vertex lies on the line segment between two positive vertices.  Thus, the convex hulls of a set of positive and negative vertices intersect non-trivially if and only if their interiors intersect.  

Thus, each point of $A$ can be moved within a small neighbourhood of its original position without changing whether a given set of positive and negative points have intersecting convex hulls, i.e., without changing the combinatorial type associated to $A$.  We can thus slightly perturb the points within this small neighbourhood to obtain a set of points in general position.  This process will yield an affine diagram $A' = A'_+ \cup A'_-$ of the same combinatorial type as $A$ but with all points in general position, along with the additional property that $|A'_+| = |A_+|$.
\end{proof}

It has been shown by Tumarkin and Felikson \cite{FT1,FT2} that for $k \geq 4$, every compact $d$-polytope with $d+k$ facets has at least two pairs of non-intersecting facets (see Theorem \ref{thm: FT one pair}).  Given a polytope $P$ with $d+4$ facets, let $f$ and $f'$ be a non-intersecting pair of facets.  In any Gale diagram associated to $P$, the points corresponding to $f$ and $f'$ can be separated from the remaining points by a hyperplane through the origin.  Take the affine Gale diagram obtained by orthogonal projection onto this hyperplane and choosing the half-space containing $f$ and $f'$ to be positive.  Applying the results of Proposition \ref{prop: gen pos}, in particular the last line of the proof, we obtain an affine Gale diagram associated to $P$ where all points are in general position and with exactly two positive points.  

\begin{prop}\label{prop: affine pos int}
Let $A = A_+ \cup A_-$ be an affine Gale diagram for a simple $d$-polytope with $d+4$ facets, where $A$ is in general position in $\R^2$ and $|A_+| = 2$.  Then $A_+$ is contained in the interior of the convex hull of $A_-$.
\end{prop}
\begin{proof}
Suppose that $u \in A_+$ is not in the interior of the convex hull of $A_-$.  Then there is some $v \in A_-$ such that no two points of $A_-$ are separated by the line through $u$ and $v$.  In other words, all points of $A_-$ lie in a closed half space $H \cup H^+$ bounded by the line $H$ through $u$ and $v$.  Observe then that the number of positive points in $H^+$ is at most one, as $|A_+| = 2$ and $u \in H$.  Moreover, there are no negative points in $H^-$.  As discussed in Subsection \ref{subsec: gale diagrams}, $|A_+ \cap H^+| + |A_- \cap H^-|$ must be at least $2$ for $A$ to be an affine Gale diagram, so we reach a contradiction.
\end{proof}

In summary, every simple $d$-polytope with $d+4$ facets admits an affine Gale diagram $A = A_+ \cup A_-$ where
\begin{enumerate}[(i)]
    \item all points of $A$ are in general position,
    \item $|A_+| = 2$, and
    \item $A_+$ is contained in the interior of the convex hull of $A_-$.
\end{enumerate}
In the sequel, we use these restrictions to show that all affine Gale diagrams of compact $d$-polytopes with $d+4$ facets can be obtained by bipartitioning the $d + k$ points of a point set order type.

\subsection{Point set order types}\label{subsec: point set}
In order to enumerate the possible Gale diagrams of compact $d$-polytopes with $d+4$ facets, we utilise a classification of the point set order types of size at most $10$ \cite{AAK}. This classification was obtained with the aid of extensive computer search in 2002 and is restricted to point sets in general position. The order type of a point set records the relative orientations of triples of points, from which one can determine many other combinatorial properties, such as whether a set of points is in convex position or whether line segments between points intersect.  

The \emph{orientation} of an ordered triple $((x_1,x_2),(y_1,y_2),(z_1,z_2)) \in (\R^2)^3$ of points in general position is determined by the sign of the determinant of the matrix 
$$\begin{bmatrix} 1 & 1 & 1\\ x_1 & y_1 & z_1\\ x_2 & y_2 & z_2\end{bmatrix}\,.$$
The points are said to be oriented \emph{counter-clockwise} if the determinant is positive and \emph{clockwise} if the determinant is negative. 
The \emph{order type} of a set of points $\{p_1,\dots,p_n\} \subset \R^2$ in general position is a mapping that assigns to each ordered triple $(i,j,k)$ of distinct elements in $\{1,\dots,n\}$ the orientation of the point triple $(p_i,p_j,p_k)$.  One can also define the order type of points not in general position, but we are not concerned with those in the present setting.  Two finite subsets $P,Q \subset \R^2$ are said to be \emph{combinatorially equivalent} if they have  the  same  order  type, i.e., if there exists a bijection $f:P\to Q$ that preserves orientations.  

The following lemma shows that the order type of a point set records the same intersection properties of convex hulls that the affine Gale diagrams use.  This is a crucial step in seeing that all affine Gale diagrams of compact $d$-polytopes with $d+4$ facets can be obtained by partitioning the points of each order type.

\begin{lem}\label{lem: order type conv int}
Fix two combinatorially equivalent point sets $P,Q \subseteq \R^2$, each in general position, with an orientation-preserving bijection $\sigma: P \to Q$ between them.  Then for any $P_1,P_2 \subseteq P$ and $Q_1 = \sigma(P_1),Q_2 = \sigma(P_2) \subseteq Q$, we have
$$\conv(P_1) \cap \conv(P_2) = \emptyset \iff \conv(Q_1) \cap \conv(Q_2) = \emptyset\,.$$
\end{lem}
\begin{proof}
Note that we can assume $P_1$ and $P_2$, hence $Q_1$ and $Q_2$, are disjoint, since otherwise both intersections are clearly non-empty.  Since the sets are in general position, $\conv(P_1)$ and $\conv(P_2)$ intersect non-trivially if and only if their interiors intersect.  If their interiors intersect, there are two possibilities, up to swapping $P_1$ and $P_2$:
\begin{enumerate}[(i)]
    \item there exist points $u,v \in P_1$ and $x,y \in P_2$ such that the line segment $\overline{uv}$ transversally intersects the line segment $\overline{xy}$, i.e, the segments intersect at an interior point, or
    \item $\conv(P_1)$ is contained in the interior of $\conv(P_2)$.
\end{enumerate}

We can detect the first condition by looking at the orientations of certain triples.  The line segments $\overline{uv}$ and $\overline{xy}$ intersect transversally if and only if $(u,v,x)$ and $(u,v,y)$ have opposite orientations, and additionally $(x,y,u)$ and $(x,y,v)$ have opposite orientations.  

The second condition can also be detected by the order type.  Recall that a set of points in $\R^2$ is in convex position, i.e., form the vertices of a convex polygon, if and only if every set of four points is in convex position (this follows directly from Caratheodory's theorem; see, e.g., \cite{Mat}).  The latter condition is equivalent to saying that every four points define precisely one pair of intersecting line segments, which can be detected by the order type as described in the previous paragraph.  Moreover, given a set of points $(p_1,\dots,p_n)$ forming cyclically adjacent vertices of a convex polygon, hence in convex position, a point $p$ is contained in their convex hull if and only if the orientations of $(p_i,p_{i+1},p)$ are all the same, where the subscripts are taken modulo $n$.  Now, if $\conv(P_1)$ is contained in $\conv(P_2)$, then there is a subset $P_2' \subset P_2$ in convex position such that $\conv(P_1) \subseteq \conv(P_2')$.  Using the order type, we can then detect if each vertex of $P_1$ is contained in the convex hull of $P_2'$, hence the convex hull of $P_2$.  

Therefore, whether $\conv(P_1) \cap \conv(P_2)$ is empty is determined by the orientations of triples in $P_1 \cup P_2$.  That is, whether each such intersection is non-empty is determined by the order type of $P$.
\end{proof}

We now obtain our method for generating the desired affine Gale diagrams using the point set order types. 

\begin{thm}\label{thm: affine generation}
Every compact Coxeter $d$-polytope $P$ with $d+4$ facets admits an affine Gale diagram obtained by taking an arrangement of $d+4$ points $A \subseteq \R^2$ in general position and choosing two points from the interior of $\conv(A)$ to be positive.  Moreover, the combinatorial type of $P$ is completely determined by the order type of $X$.
\end{thm}
\begin{proof}
This follows directly from Proposition \ref{prop: gen pos}, Proposition \ref{prop: affine pos int}, and Lemma \ref{lem: order type conv int}.
\end{proof}

An example of such an affine Gale diagram and a polytope realising its combinatorial type are depicted in Figure \ref{fig: affine G6}.

\begin{figure}[ht]
\begin{center}
\includegraphics[width=.7\linewidth]{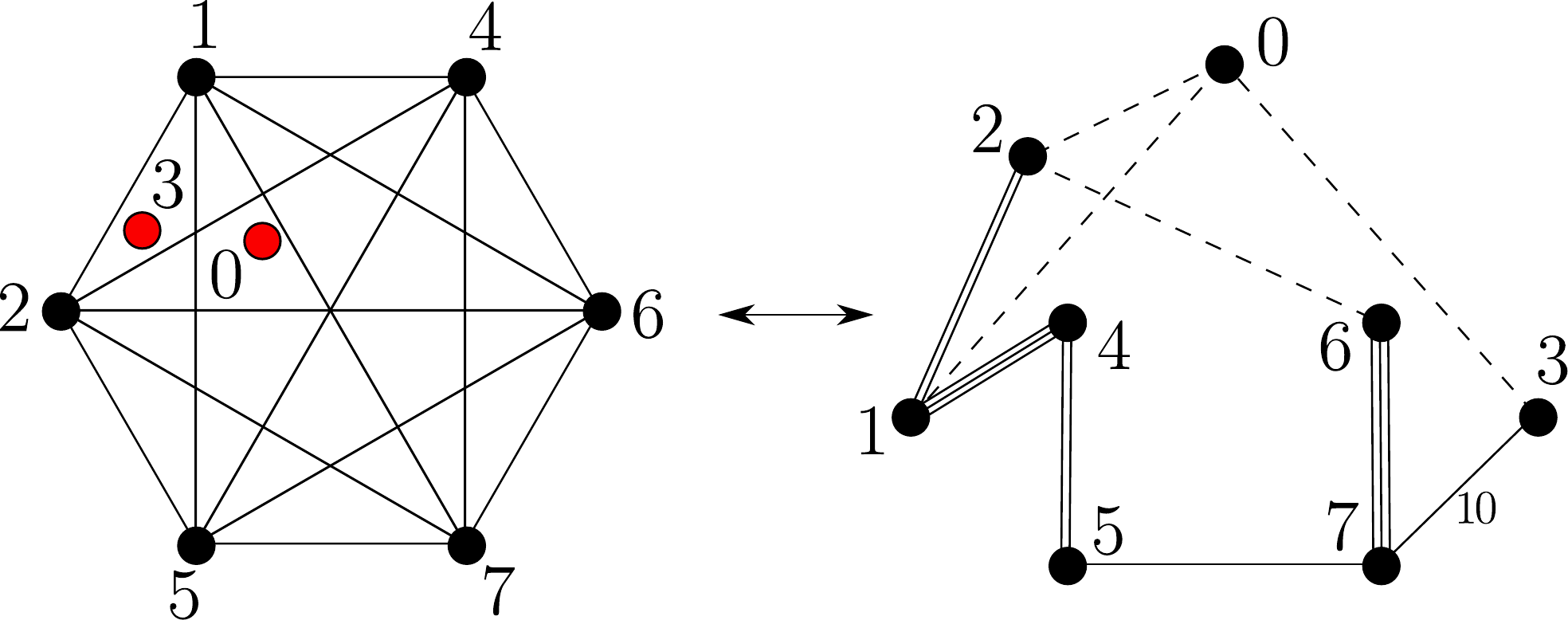}
\caption{On the left is an affine Gale diagram of combinatorial type $G_6$, with the positive points depicted in red and the negative points in black.  The line segments between negative points have been included to aid in identifying the convex hull of each subset.  On the right, we have a Coxeter diagram of a compact $4$-polytope with combinatorial type $G_6$ (see Appendix \ref{app: Gale diagrams} for a description of this type); the labelling of the vertices corresponding to the $8$ facets is preserved (see Section \ref{sec: Coxeter and Gram} for details of how to interpret Coxeter diagrams).}
\label{fig: affine G6}
\end{center}  
\end{figure}

Using the order type database for up to $10$ points \cite{AAK}, one can classify the combinatorial types of all polytopes with $d + 4$ facets in dimension at most $6$.  There are $3{,}315$ order types on $8$ points, which yield $34$ possible combinatorial types of simple $4$-polytopes with $8$ facets having at least one pair of disjoint facets, listed in Table \ref{tab: d+4 comb types}.  The classification in dimension $4$ was also completed by Gr{\" u}nbaum and Sreedharan \cite{GSr} via different methods, though they also considered polytopes with no pairs of disjoint facets.  There are $158{,}817$ order types on $9$ points, which yield $186$ possible combinatorial types of $5$-polytopes with $9$ facets; the $111$ combinatorial types with at least two pairs of disjoint facets are listed in Appendix \ref{app: Gale diagrams}.  Note that not all of these combinatorial types are realized; this is merely a short list of possible types from which we determine those that are realized as compact Coxeter polytopes.  

There are $14{,}309{,}547$ order types on $10$ facets; the significant jump in the number of order types, recalling that we must also iterate over all choices of two positive vertices for each order type, makes determining the possible combinatorial types for this class computationally challenging.  The author plans to handle this class in a later work.  However, we have been able to glean some partial information about $6$-polytopes with 10 facets; see Section \ref{sec: dim 6 polys} for more details.

\section{Coxeter Diagrams and Gram Matrices}\label{sec: Coxeter and Gram}
In this section, we define (abstract) Coxeter diagrams and Gram matrices.  We also illustrate several properties of these structures when associated to a compact hyperbolic Coxeter $d$-polytope with $d+4$ facets, which are used in our classification of these polytopes in the subsequent sections.

We begin by introducing the notion of an \emph{abstract Coxeter diagram}, a weighted graph which can, under certain conditions, describe the dihedral angles of a polytope.  

\begin{defn}
An \emph{abstract Coxeter diagram} is a finite simple graph (i.e., one-dimensional simplicial complex) with weighted edges.  The weights $w_{ij}$ must be positive and satisfy the following condition: if $w_{ij} < 1$, then $w_{ij} = \cos\left( \frac{\pi}{m_{ij}}\right)$ for some integer $m_{ij} \geq 3$. 
\end{defn}

We often denote abstract Coxeter diagrams by the letter $\Sigma$.  The order $|\Sigma|$ of the diagram $\Sigma$ is the number of vertices of $\Sigma$.  A \emph{subdiagram} $\Sigma'$ of $\Sigma$ is a vertex-induced subgraph of $\Sigma$, where the edge weights are preserved; this relation is written as $\Sigma' \subset \Sigma$.  For any two subdiagrams $\Sigma_1,\Sigma_2$ of $\Sigma$, we let $\langle \Sigma_1, \Sigma_2\rangle$ denote the subdiagram induced by the vertices contained in either $\Sigma_1$ or $\Sigma_2$.  When drawing abstract Coxeter diagrams, we adhere to the following conventions:
\begin{enumerate}[(i)]
    \item If $w_{ij} < 1$, hence $w_{ij} = \cos\left( \frac{\pi}{m_{ij}}\right)$ for some integer $m_{i,j} \geq 2$, then the corresponding edge is drawn as a straight line labelled by $m_{ij}$.  In the special cases where $m_{ij}$ is equal to $3$, $4$, or $5$, we draw the corresponding edge as an unlabelled single, double, or triple line, respectively. If $m_{ij} = 2$, we leave the edge empty.
    \item If $w_{ij} > 1$, the corresponding edge is drawn as a dashed line with label $w_{ij}$. 
    \item If $w_{ij} = 1$, the corresponding edge is drawn as an unlabelled bold line.  Since we are considering only compact polytopes, this case does not arise in our setting.
\end{enumerate}

An ordinary edge of weight $\cos\left( \frac{\pi}{m}\right)$ for some $m \geq 2$ is said to have \emph{multiplicity} $m-2$.  We refer to an ordinary edge of weight $ \cos\left( \frac{\pi}{m_{ij}}\right)$ for $2 \leq m_{ij} \leq 5$ as having \emph{low weight}, and an ordinary edge of weight $\cos\left( \frac{\pi}{m_{ij}}\right)$ for $m_{ij} \geq 6$, i.e., having multiplicity at least $4$, as being \emph{multi-multiple}.  While abstract Coxeter diagrams encode the information of the dihedral angles in a graph theoretic context, it is also often useful to treat these linear-algebraically.  We do so by considering a certain weighted adjacency matrix, known as a Gram matrix.

\begin{defn}
The \emph{Gram matrix} $M(\Sigma)$ of an abstract Coxeter diagram $\Sigma$ on $n$ vertices is an $n \times n$ matrix with entries $m_{ij}$ as prescribed:
$$m_{ij} = \begin{cases} 1 & \text{ if } i = j\,,\\
                        -w_{ij} & \text{ if vertices } i,j \in \Sigma \text{ are connected by an edge of weight $w_{ij}$}\,,\\
                        0 & \text{ if vertices } i,j \in \Sigma \text{ are not adjacent}\,.
                        \end{cases}$$
\end{defn}

Now that we have introduced the notions of Coxeter diagrams and Gram matrices abstractly, we describe their relation to Coxeter polytopes.  The \emph{Coxeter diagram} $\Sigma(P)$ of $P$ is a weighted graph with vertices corresponding to facets $\{f_i\}_{i \in I}$, where two vertices corresponding to $f_i$ and $f_j$ are connected by

\begin{itemize}
\item $\text{an ordinary edge of weight $\cos(\frac{\pi}{k})$ if $f_i$ and $f_j$ meet at angle $\frac{\pi}{k}$;}$
\item $\text{a dashed edge of weight $\cosh(\rho)$ if $f_i$ and $f_j$ diverge and are at distance $\rho$ apart;}$
\item $\text{a bold edge of weight $1$ if $f_i$ and $f_j$ are parallel.}$
\end{itemize}

The \emph{Gram matrix} $M(P)$ of a Coxeter polytope $P$ is the Gram matrix of the corresponding Coxeter diagram, i.e., $M(\Sigma(P))$.  Note that $M(P)$ coincides with the Gram matrix of unit normal vectors to the facets of $P$.

Recall that the \emph{signature} of a real symmetric matrix is a triple $(n_+, n_-,n_0)$ of non-negative integers where $n_+$ is the number of positive eigenvalues (the \emph{positive inertia index}), $n_-$ is the number of negative eigenvalues (the \emph{negative inertia index}), and $n_0$ is the multiplicity of $0$ as an eigenvalue, i.e., the nullity.  
Thus, a real symmetric matrix is \emph{positive definite} if it has signature $(n,0,0)$.  A matrix is called \emph{indecomposable} if it cannot be transformed into a block-diagonal matrix via simultaneous permutations of columns and rows. 

An abstract Coxeter diagram $\Sigma$ is called 
\begin{itemize}
    \item \emph{elliptic} if $M(\Sigma)$ is positive definite (the connected elliptic diagrams are listed in Figure \ref{fig: ell diagrams});
    \item \emph{parabolic} if any indecomposable component of $M(\Sigma)$ is singular, and every proper subdiagram of an indecomposable component is elliptic;
    \item \emph{Lann{\'e}r} if $\Sigma$ is connected and is neither elliptic nor parabolic, and any proper subdiagram of $\Sigma$ is elliptic (all Lann{\'e}r diagrams are listed in Figure \ref{fig: Lanner diagrams});
    \item \emph{hyperbolic} if $M(\Sigma)$ has negative inertia index $n_- = 1$;
    \item \emph{superhyperbolic} if $M(\Sigma)$ has negative inertia index $n_- > 1$;
    \item \emph{admissible}, following the terminology from \cite{Tum}, if $M(\Sigma)$ is not superhyperbolic and contains no parabolic subdiagrams.
\end{itemize}

\begin{figure}[ht]
\begin{center}
\includegraphics[width=.95\linewidth]{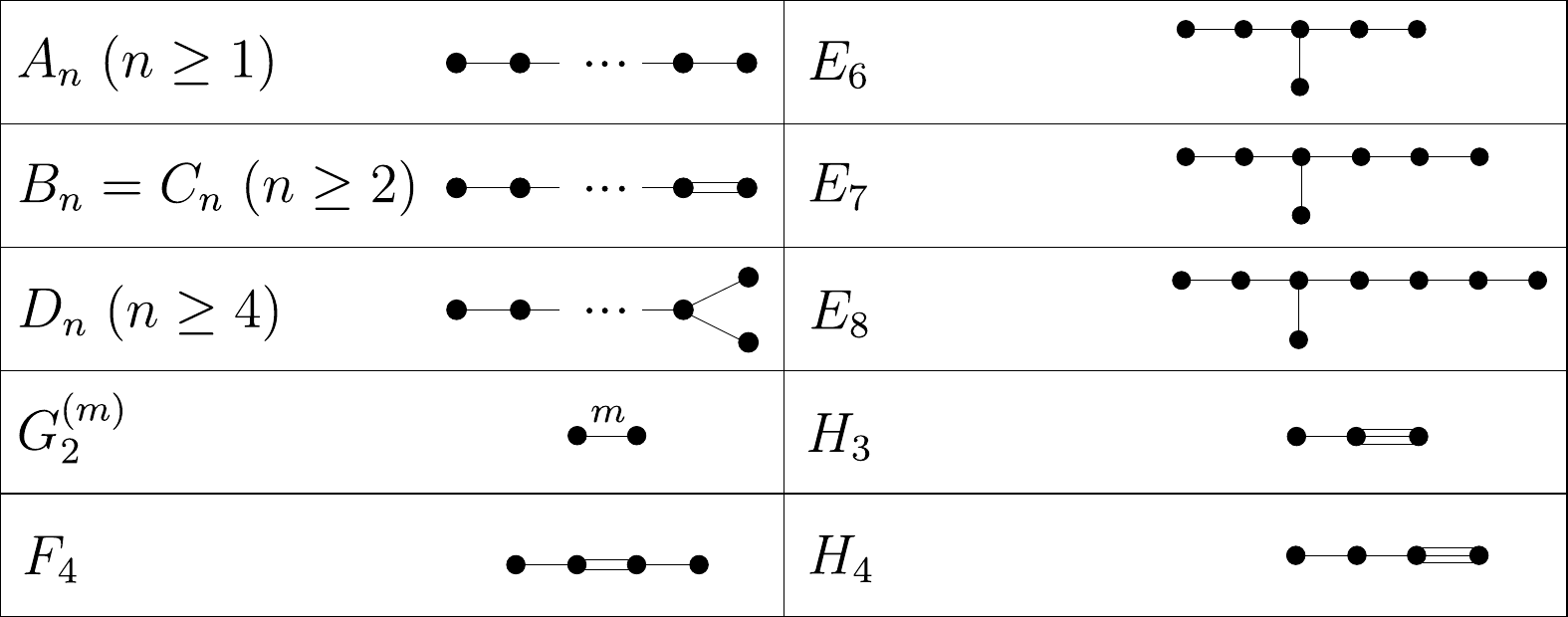}
\end{center}  
\label{fig: ell diagrams}
\caption{The list of all connected elliptic diagrams along with their type, where the subscript corresponds to the number of vertices.}
\end{figure}

\begin{figure}[ht]
\begin{center}
\includegraphics[width=.6\linewidth]{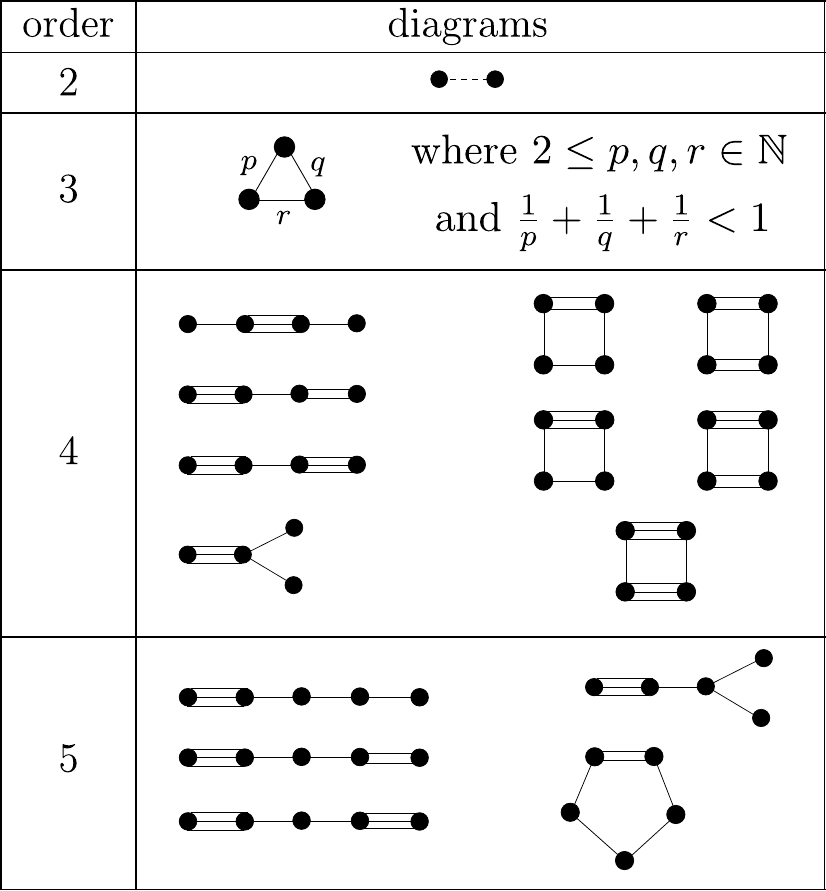}
\end{center}  
\caption{The list of all Lann{\'e}r diagrams, classified by Lann{\'e}r in \cite[Table 3]{Lan}.}
\label{fig: Lanner diagrams}
\end{figure}

\begin{rem}\label{rem: Vinberg conditions}
In \cite{Vin}, Vinberg shows that if $\Sigma = \Sigma(P)$ is the Coxeter diagram of a compact hyperbolic $d$-polytope $P$, then $\Sigma$ is an admissible connected hyperbolic diagram with positive inertia index $d$.  Note that this implies that $\Sigma$ does not contain any parabolic subdiagrams nor bold edges.
\end{rem}    

We can see from the above remark that  $M(\Sigma)$ has rank $d+1$, so any principal minor of order $d+2$ or larger is singular.  Though we often omit the weights of dashed edges when drawing Coxeter diagrams, these can be determined from the remaining weights by applying this rank condition.  The face structure of $\Sigma(P)$ can be determined by the following proposition.

\begin{prop}[\protect{\hspace{1sp}\cite[Proposition 3.1]{Vin}}]\label{prop: Vin 3.1}
Let $P = \bigcap_{i\in I}f^-_i \subset \Hy^d$ be an acute-angled polytope and $M$ its Gram matrix.  Let $J \subset I$ be a subset such that the induced submatrix $M_J$ is positive semidefinite.  Then $P^J = P \cap \left(\bigcap_{i \in J}f_i\right)$ is a face of $P$ if and only if $M_J$ is positive definite.  If so, $P^J$ is a face of codimension $|J|$.
\end{prop}

Thus, the faces of codimension $m$ in $P$ correspond to the elliptic subdiagrams of order $m$ in $\Sigma(P)$, and hence the missing faces of $P$ correspond to the Lann{\'e}r subdiagrams of $\Sigma(P)$.  In particular, a compact hyperbolic polytope has no missing faces of order greater than $5$.  Thus, using the data of the positive semidefinite submatrices of a Gram matrix, one can determine the combinatorial type of the corresponding polytope. 

\begin{rem}\label{rem: simple Coxeter}
Note that Proposition \ref{prop: Vin 3.1} implies that every compact hyperbolic Coxeter $d$-polytope is simple, i.e., every $(d-j)$-face is contained in precisely $j$ facets.
\end{rem}

\subsection{Polytope faces and Coxeter subdiagrams}

Given a face of a compact Coxeter polytope, we discuss restrictions on the combinatorial structure of the face and, if the face is a Coxeter polytope, the structure of its Coxeter diagram.  Let $P$ be a compact Coxeter $d$-polytope with Coxeter diagram $\Sigma = \Sigma(P)$, and $S_0$ an elliptic subdiagram of order $m$.  By Proposition \ref{prop: Vin 3.1}, $S_0$ corresponds to a face of codimension $m$ in $P$; denote this face by $P(S_0)$.  While $P(S_0)$ is necessarily an acute-angled polytope, it may not be Coxeter in certain cases.  Borcherds provided the following sufficient condition for $P(S_0)$ to be Coxeter.

\begin{prop}[\protect{\hspace{1sp}\cite[Example 5.6]{Bor}}]\label{prop: Coxeter face}
Suppose $P$ is a Coxeter polytope with diagram $\Sigma$, and $S_0 \subset \Sigma$ is an elliptic subdiagram containing no connected component of type $A_n$ or $D_5$. Then $P(S_0)$ is a Coxeter polytope.
\end{prop}

We now examine what can be said about the Coxeter diagram of $P(S_0)$ for $S_0$ satisfying the conditions of Proposition \ref{prop: Coxeter face}.  Supposing that $P(S_0)$ is indeed Coxeter, denote its Coxeter diagram by $\Sigma_{S_0}$. The following terminology was introduced in \cite[Section 1]{FT}, which contains further details of the computation of the dihedral angles in $P(S_0)$.  A vertex of $\Sigma$ \emph{attaches} to $S_0$ if it is joined with any vertex of $S_0$ by an edge of any type, in which case $w$ is called a \emph{neighbour} of $S_0$.  Then $w$ is called a \emph{good neighbour} if $\langle S_0, w\rangle$ is elliptic, and \emph{bad} otherwise.  Note that in the latter case $\langle S_0, w \rangle$ must contain a Lann{\'e}r diagram. 

Let $\overline{S_0}$ denote the subdiagram of $\Sigma$ induced by the vertices corresponding to facets of $P(S_0)$, i.e., the good neighbours of $S_0$ along with all vertices not attached to $S_0$.  We now use the following results of Felikson and Tumarkin, based on the analysis of Allcock \cite[Theorem 2.2]{All}, to describe the possible differences between the diagrams $\Sigma_{S_0}$ and $\overline{S_0}$.  A \emph{simple edge} refers to an ordinary edge of weight $\cos\left(\frac{\pi}{3}\right)$, denoted by an ordinary unlabelled edge of multiplicity $1$ in a Coxeter diagram.  

\begin{prop}[\protect{\hspace{1sp}\cite[Corollary 1.1]{FT}}]
Under the hypotheses of Proposition \ref{prop: Coxeter face},
\begin{enumerate}[(a)]
    \item If $S_0$ is of the type $H_4$, $F_4$, or $G_2^{(m)}$ for $m \geq 6$, or any other diagram having no good neighbours, then $\overline{S_0} = \Sigma_{S_0}$.
    \item If $S_0$ is of the type $H_3$, then $\overline{S_0}$ may be obtained by replacing some dashed edges by ordinary edges. 
    \item If $S_0$ is of the type $G_2^{(5)}$, then $\overline{S_0}$ may be obtained from $\Sigma_{S_0}$ by replacing some edges labelled by $10$ by simple edges, and some dashed edges by ordinary edges.
    \item If $S_0$ is of the type $B_n$ for $n \geq 3$, then $\overline{S_0}$ may be obtained from $\Sigma_{S_0}$ by replacing some double edges by simple edges, and some dashed edges by ordinary edges.
    \item If $S_0$ is of the type $B_2 = G_2^{(4)}$, then $\overline{S_0}$ may be obtained from $\Sigma_{S_0}$ by replacing some double edges by simple edges, and some dashed edges by ordinary or empty edges.
\end{enumerate}
\end{prop}

We make frequent use of these restrictions in the second portion of this paper, when we bound the dimension of polytopes with few facets.

\subsection{Local determinants}
A technical tool that can help to identify superhyperbolic subdiagrams is the local determinant, for which the theory was developed in \cite{Vin2}.  Let $\Sigma$ be an abstract Coxeter diagram, and let $T$ be a subdiagram of $\Sigma$ such that $\det(\Sigma \cut T)\neq 0$.
The \emph{local determinant} of $\Sigma$ on the subdiagram $T$ is given by 
$$\det(\Sigma,T) = \frac{\det(\Sigma)}{\det(\Sigma \cut T)}\,.$$  
When $\Sigma$ is composed of two subdiagrams joined in a simple way, we have the following two results to simplify the calculation of the local determinant.

\begin{prop}[\protect{\hspace{1sp}\cite[Proposition 12]{Vin2}}]\label{prop: local vertex}
If a Coxeter diagram $\Sigma$ consists of two subdiagrams $\Sigma_1$ and $\Sigma_2$ having a unique vertex $v$ in common, and if no vertex of $\Sigma_1$ attaches to $\Sigma_2$, then
$$\det(\Sigma,v) = \det(\Sigma_1,v) + \det(\Sigma_2,v) - 1\,.$$
\end{prop}

\begin{prop}[\protect{\hspace{1sp}\cite[Proposition 13]{Vin2}}]\label{prop: local edge}
If a Coxeter diagram $\Sigma$ is spanned by two disjoint subdiagrams $\Sigma_1$ and $\Sigma_2$ joined by a unique edge $v_1v_2$ of weight $a$, then
$$\det(\Sigma,\langle v_1,v_2\rangle) = \det(\Sigma_1,v_1)\det(\Sigma_2,v_2) - a^2\,.$$
\end{prop}

In particular, under the hypotheses of Proposition \ref{prop: local edge}, if $M(\Sigma)$ is singular and $\Sigma \cut \{v_1,v_2\}$ is elliptic, then we must have $\det(\Sigma_1,v_1)\det(\Sigma_2,v_2) = a^2$.  We are particularly interested in the case when $\Sigma_1$ and $\Sigma_2$ are Lann{\' e}r triangles.  It is straightforward to check that the local determinant of a Lann{\' e}r triangle (which is necessarily negative) increases in magnitude as a function of its edge weights (see, e.g., \cite{Vin2}).

\section{\texorpdfstring{Properties of Compact Hyperbolic $d$-Polytopes with $d+4$ Facets}{Properties of Compact Hyperbolic d-Polytopes with d+4 Facets}}\label{sec: d+4 properties}

Before delving into the characterisation of compact hyperbolic Coxeter $d$-polytopes with $d+4$ facets, we first mention several restrictions on these polytopes that we refer to throughout the classification.

The first restrictions we consider are related to the number of pairs of disjoint facets.  The compact Coxeter polytopes with no pairs of disjoint facets are precisely the $d$-simplices, which have $d+1$ facets and were classified by Lann{\'e}r \cite{Lan}, along with the seven Esselmann polytopes, which are $4$-polytopes with $6$ facets that were constructed by Esselmann in \cite{Ess2}; this list was shown to be complete by Felikson and Tumarkin \cite[Theorem A]{FT1}.  In 2009, Felikson and Tumarkin studied compact Coxeter polytopes with precisely one pair of disjoint facets.

\begin{thm}[\protect{\hspace{1sp}\cite[Main Theorem]{FT2}}]\label{thm: FT one pair}
 A compact hyperbolic Coxeter $d$-polytope with exactly one pair of non-intersecting facets has at most $d + 3$ facets.
\end{thm}

Thus, in the setting of $d$-polytopes with $d+4$ facets, we can restrict to looking at polytopes with at least two pairs of disjoint facets.  This can be strengthened in dimension $4$.  Felikson and Tumarkin \cite{FT3} later studied $d$ polytopes with $n$ facets having at most $n-d-2$ pairs of disjoint facets.  They gave a finite algorithm for listing these polytopes, and carried out this algorithm in dimension $4$.  The algorithm produced no previously-unknown polytopes of dimension $4$, yielding the following theorem.

\begin{thm}[\protect{\hspace{1sp}\cite[Theorem 7.1]{FT3}}]\label{thm: FT ess}
 Any compact hyperbolic Coxeter $4$-polytope with $n$ facets having at most $n - 6$ pairs of disjoint facets satisfies $n \leq 7$.
\end{thm}

For the special case $d=4$ and $n = 8$, we immediately obtain the following result.

\begin{cor}\label{cor: 4 8 disjoint}
Any compact hyperbolic Coxeter $4$-polytope with $8$ facets must have at least three pairs of disjoint facets.
\end{cor}  

As we shall see in Section \ref{sec: d+4 class}, there do exist such polytopes with exactly three pairs of disjoint facets, in particular, those of combinatorial type $G_{10}$, $G_{12}$, or $G_{14}$ (see Appendix \ref{app: Gale diagrams} for their description).

Following the methods of Tumarkin in \cite{Tum}, we can use a geometric property of polytopes to reduce our search.  If a polytope has a facet $f$ that meets precisely $d$ other facets, we call this facet a \emph{prism facet}, and the corresponding vertex of the Coxeter diagram a \emph{prism vertex}.  Note that this condition is equivalent to $f$ being a simplex, since as a $(d-1)$-polytope it would have precisely $d$ facets.  We additionally refer to all edges incident to a prism vertex as \emph{prism edges}.  At each prism facet $f$, we can truncate the polytope by a hyperplane $H$ to obtain a new polytope $P'$ with the following properties:
\begin{itemize}
    \item The new facet $f' = P \cap H$ of $P'$ does not intersect $f$,
    \item $f'$ is either disjoint from or orthogonal to each facet of $P'$, and
    \item $P'$ is combinatorially equivalent to $P$.
\end{itemize}

In other words, we can obtain $P$ from a combinatorially equivalent polytope $P'$, where all facets of $P'$ corresponding to prism facets of $P$ meet any incident facets at dihedral angle $\frac{\pi}{2}$, by gluing Coxeter prisms onto $P'$.  The hyperplane $H$ can be obtained by choosing successive facets of $f$ and transforming the bounding hyperplane along each chosen facet of $f$ so that they meet at dihedral angle $\frac{\pi}{2}$.  This process terminates with $P'$ being combinatorially equivalent to $P$ since each transformation preserves the face structure, as $P$ is acute-angled.  In our classification, we first assume prism facets already satisfy the dihedral angle condition, and then later obtain any combinatorially equivalent polytopes by gluing compact Coxeter prisms.  The compact Coxeter prisms have been classified by Kaplinskaja \cite{Kap}.

Lastly, we mention a rank condition on the Gram matrix.  As discussed in Section \ref{sec: Coxeter and Gram}, the Gram matrix of a compact Coxeter $d$-polytope has rank $d + 1$.  This is because the normal vectors to the facets are in $\Hy^d$, which can be embedded in $\R^{d+1}$.  Thus, we have the following condition on the principal minors.

\begin{prop}\label{prop: rank condition}
If $P$ is a compact Coxeter $d$-polytope with $d+4$ facets, then every subdiagram $\Sigma_0$ obtained by deleting two vertices of $\Sigma(P)$ has a singular Gram matrix.
\end{prop}

\subsection{The set of multi-multiple edges}\label{subsec: multi-multiple}
In this section, we mention several restrictions on which ordinary edges can be multi-multiple, i.e., have weight $\cos\left(\frac{\pi}{m}\right)$ for $m \geq 6$, in a compact Coxeter diagram.  Restricting this set is critical to making the classification in Sections \ref{sec: d+4 class} and \ref{sec: d+5 class} computationally feasible.  

\begin{lem}\label{lem: multi-multiple}
Let $v_1v_2$ be an ordinary edge in an admissible abstract Coxeter diagram $\Sigma$.  Suppose that
\begin{enumerate}[(a)]
    \item $\Sigma$ has a Lann{\'e}r diagram of order greater than $3$ containing $v_1$ and $v_2$; or
    \item $\Sigma$ has no Lann{\'e}r diagram containing $v_1$ and $v_2$, $\Sigma$ has a Lann{\'e}r diagram $L$ of order greater than $2$ containing $v_2$, and there is no dashed edge from $v_1$ to any vertex of $L$\,.
\end{enumerate}
Then $v_1v_2$ has low weight.
\end{lem}
\begin{proof}
The first part is an immediate consequence of the characterisation of Lann{\'e}r diagrams given in Figure \ref{fig: Lanner diagrams}.  

For the second part, note that $v_2$ must be incident to an ordinary edge induced by $L$; call this edge $v_2v_3$.  If $v_1v_2$ was multi-multiple, then $\{v_1,v_2,v_3\}$ would induce a parabolic diagram or Lann{\' e}r triangle, since $v_1v_3$ is not dashed by assumption.  The former cannot happen because $\Sigma$ is admissible, and the latter cannot happen because we assume $\Sigma$ has no Lann{\'e}r diagram containing both $v_1$ and $v_2$.  
\end{proof}

\begin{lem}\label{lem: overlapping triangles}
Suppose that $01234$ is a subdiagram of an admissible abstract Coxeter diagram.  If the induced set of missing faces is $\{012,234\}$, then $02$, $12$, $23$, and $24$ are not multi-multiple.
\end{lem}
\begin{proof}
Note that one of $23$ or $24$ must be non-empty in order for $234$ to be a missing face of size $3$; without loss of generality suppose $23$ is non-empty.  Thus, if either $02$ or $12$ is multi-multiple, it would induce a missing face $023$ or $123$, which is forbidden.  Hence neither $02$ nor $12$ is multi-multiple; the analogous result for $23$ and $24$ can be obtained by swapping the roles of $0$,$1$ and $3$,$4$, respectively.
\end{proof}

\begin{lem}[\protect{\hspace{1sp}\cite[Lemma 4.14]{Tum}}]\label{lem: Tum forbidden subdiagram}
There is no compact Coxeter $4$-polytope containing a subdiagram with induced missing face list isomorphic to $\{0123, 014, 235\}$.
\end{lem}

\subsection{Admissible partial weightings}
Now that we have discussed some general properties of compact Coxeter $d$-polytopes with $d+4$ facets and the multi-multiple edges of their Coxeter diagrams, we focus on the possible structure of the ordinary low-weight edges of their Coxeter diagrams.  
\begin{defn}
We define the \emph{partial low-weighting} of an abstract Coxeter diagram $\Sigma$ as the image $\Sigma^{\leq 6}$ of a forgetful map $\phi$ to a new weighted diagram on the same vertex set where all edge weights of the form $\cos\left(\frac{\pi}{m}\right)$ for $m \geq 6$ or $\cosh(\rho)$ for any $\rho \in \R$ are forgotten, though the information of whether these edges are ordinary or dashed remains.  For a multi-multiple edge whose weight was forgotten, we denote its new weight by $*$.
\end{defn}

Figure \ref{fig: low-weighting example} depicts the partial low-weighting of a Coxeter diagram, the same one shown in Figure \ref{fig: affine G6}.

\begin{figure}[ht]
\begin{center}
\includegraphics[width=.8\linewidth]{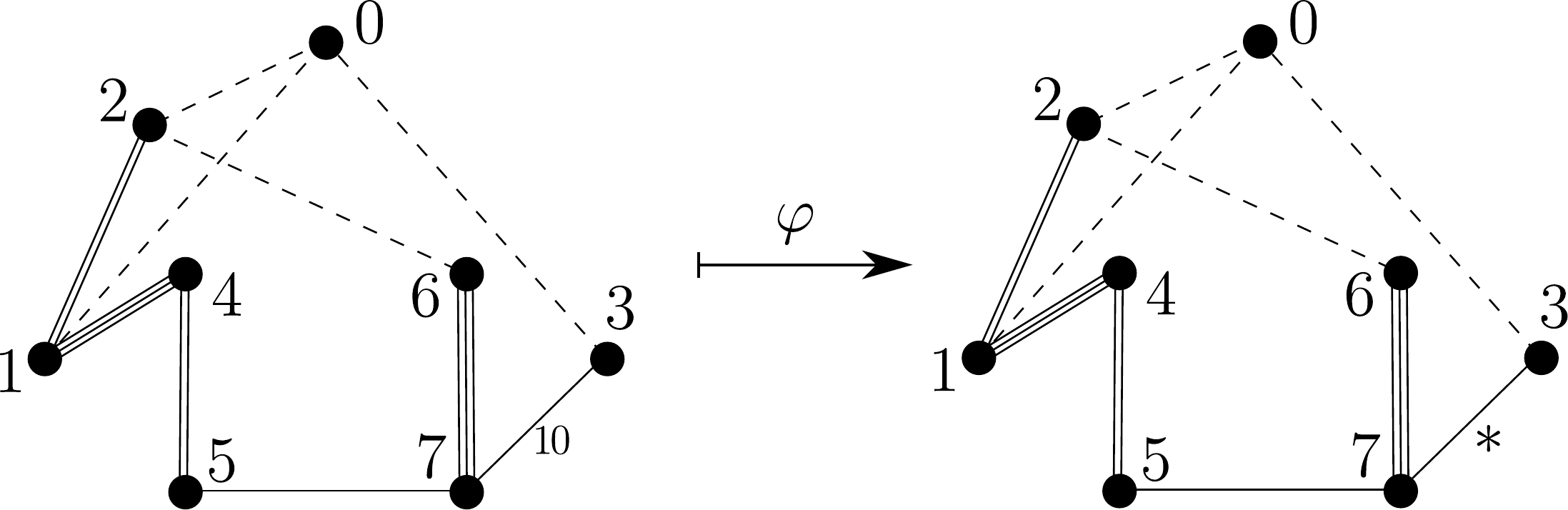}
\caption{This illustrates the action of the forgetful map $\phi$ on an abstract Coxeter diagram $\Sigma$ (on the left) to obtain its partial low-weighting (on the right).  Note that the unique multi-multiple edge changes weight from $10$ to $*$.  Moreover, the dashed edges in the left diagram have weight induced by Proposition \ref{prop: rank condition}, while the dashed edges on the right have no associated weight.}
\label{fig: low-weighting example}
\end{center}  
\end{figure}

Note that there are finitely many partial low-weightings of a given order, as each edge is either simple, double, triple, dashed, or multi-multiple with weight $*$.  The forgetful map $\phi$ also retains some information about the  elliptic subdiagrams: we say that a subdiagram $S_0$ of a partial low-weighting is \emph{elliptic} if it remains elliptic when replacing the edges weighted by $*$ with any weights of the form $\cos\left(\frac{\pi}{m_{ij}}\right)$ for $m_{ij} \geq 6$.  Since the only connected elliptic diagrams containing a multi-multiple edge are $G_2^{(m)}$ for $m \geq 6$, this is well-defined.  In particular, a subdiagram $S_0$ of a partial low-weighting is elliptic if and only if any connected component of $S_0$ is either comprised of ordinary edges and is elliptic, or consists of a single edge with weight $*$.  A \emph{missing face} of a partial low-weighting is a set of vertices $S_0$ (corresponding to a set of facets) such that the subdiagram $S_0$ is not elliptic, but every proper subdiagram is elliptic.  

We say that a partial low-weighting $\Sigma_1^{\leq 6}$  \emph{has the same combinatorial type} as (the partial low-weighting of) a Coxeter diagram $\Sigma_2$ if their sets of missing faces are the same, up to a relabelling of the vertices.  We call a partial low-weighting on the vertex set $\{0,1,\dots,d+k-1\}$ \emph{admissible} if it can be obtained as the image of an admissible abstract Coxeter diagram.  In particular, note that an \emph{admissible} partial low-weighting must not contain any parabolic subdiagrams.

Though we initially define partial low-weightings as the images of Coxeter diagrams under the map $\phi$, we can conversely consider constructing an abstract Coxeter diagram from a low-partial weighting.  We call an abstract Coxeter diagram $\Sigma_1$ an \emph{admissible extension} of a partial low-weighting $\Sigma_2^{\leq 6}$ if $\Sigma_1$ is admissible and $\Sigma_2^{\leq 6} = \Sigma_1^{\leq 6}$.

A critical step in our method is ensuring that we have a system of polynomial restrictions on our abstract Coxeter diagrams such that any valid partial low-weighting is the image of at most one admissible Coxeter diagram satisfying the polynomial restrictions.  For each combinatorial type $G_i$, we denote the corresponding system of equations by $V(G_i)$.  In our case, each system is determined by checking that certain principal submatrices of $M(\Sigma)$ of size at least $d+2$ are singular, which is satisfied for any abstract Coxeter diagram corresponding to a $d$-polytope by Proposition \ref{prop: rank condition}.  For an $n\times n$ matrix $M$ and $A,B \subseteq \{0,1,\dots,n-1\}$, let $M|_{A,B}$ denote the submatrix formed by restricting to the rows with indices in $A$ and the columns with indices in $B$ (where rows and columns are zero-indexed).  Thus, $M|_{A,B}$ is a $|A| \times |B|$ submatrix.  The elements of $V(G_i)$ are represented by strings $i_1\cdots i_m$, representing the equation 
$$\det(M|_{S,S}) = 0 \text{ where } S = \{0,\dots,n-1\} \cut \{i_1,i_2,\dots,i_m\}\,.$$
Note that while this conflicts with our notation for subdiagrams/facets, these strings are only ever written as members of a set $V(G)$ for some combinatorial type $G$, so the usage is made clear by the context.

\begin{exmp}
Fix a $4 \times 4$ matrix $M = (a_{i,j})_{0\leq i,j\leq 3}$.  Then using our notation, $V(G) = \{01,13\}$ is the system of equations
\[
\begin{cases}
\det\left(\begin{bmatrix} a_{2,2} & a_{2,3} \\ a_{3,2} & a_{3,3}\end{bmatrix}\right) = 0\,,\\\\
\det\left(\begin{bmatrix} a_{0,0} & a_{0,2} \\ a_{2,0} & a_{2,2}\end{bmatrix}\right) = 0\,,
\end{cases}
\]
or, equivalently,
\[
\begin{cases}
a_{2,2}a_{3,3} - a_{2,3}a_{3,2} = 0\,,\\
a_{0,0}a_{2,2} - a_{0,2}a_{2,0} = 0\,.
\end{cases}
\]
\end{exmp}

We now discuss our main method for classifying compact Coxeter polytopes, which involves finding a sufficient set of restrictions such that all possible compact Coxeter polytopes of a given combinatorial type can be listed by a finite algorithm.

\begin{prop}\label{prop: iterate over partial}
Fix a combinatorial type $G$, a set of restrictions $R(G)$ on the dihedral angles, and a system of equations on the edge weights $V(G)$.  Suppose $V(G)$ is chosen such that for every admissible partial low-weighting of type $G$ satisfying $R$, there are finitely many solutions to $V(G)$.  Then all Coxeter diagrams corresponding to compact polytopes of type $G$ satisfying $V(G)$ and $R(G)$ can be listed by iterating over all admissible partial low-weightings of $G$ and testing whether the solutions to $V(G)$ yield a diagram satisfying the properties of Remark \ref{rem: Vinberg conditions}.
\end{prop}

In our classification, we first fix a combinatorial type $G$ and then deduce certain properties of the dihedral angles that must hold for any compact Coxeter polytope of type $G$.  This set of restrictions will comprise $R(G)$.  Moreover, our set of equations $V(G)$ will be a subset of the restrictions guaranteeing that the Gram matrix of a $d$-polytope has rank $d+1$. Hence, $V(G)$ is satisfied by the Gram matrix of any compact Coxeter $d$-polytope.  Imposing the conditions of $R(G)$ and $V(G)$ does not eliminate any compact Coxeter polytopes of combinatorial type $G$.  Thus, we obtain the following corollary, stating that Proposition \ref{prop: iterate over partial} yields all compact Coxeter polytopes of type $G$ when $R(G)$ and $V(G)$ are as described.

\begin{cor}\label{cor: pass to computation}
Suppose $G$, $R(G)$, and $V(G)$ are chosen satisfying the hypotheses of Proposition \ref{prop: iterate over partial}, and additionally that all equations in $V(G)$ come from setting principal minors of size $\geq d+2$ to zero, and that the properties in $R(G)$ are satisfied for all compact Coxeter polytopes of type $G$.  Then the process in Proposition \ref{prop: iterate over partial} yields all Coxeter diagrams corresponding to polytopes of type $G$. 
\end{cor}
\begin{proof}
Since the rank of the Gram matrix $M(P)$ for any compact Coxeter $d$-polytope $P$ is $d+1$, then every principal minor of size $d+2$ or larger must vanish.  Thus, all Coxeter diagrams corresponding to compact polytopes of type $G$ satisfy $V(G)$.  By assumption, all compact Coxeter polytopes of type $G$ also satisfy $R(G)$.  The process in Proposition \ref{prop: iterate over partial} yields the set of all compact Coxeter polytopes of type $G$ satisfying $V(G)$ and $R(G)$, which is merely equivalent to the set of all compact Coxeter polytopes of type $G$.
\end{proof}

We emphasize that in Proposition \ref{prop: iterate over partial}, our conditions on $V(G)$ require iterating over all admissible partial low-weightings of type $G$ satisfying $R(G)$.  In practice, this involves checking that the system $V(G)$ has finitely many solutions for each such low-weighting, which can be done with a computer algebra system such as {\tt Mathematica}.  We often choose $V(G)$ to be minimal or nearly minimal to satisfy the hypotheses of Proposition \ref{prop: iterate over partial}, since in practice this speeds up the necessary computations.  However, we are not aware of a less computationally challenging method for checking the sufficiency of the choice of $V(G)$, as the system may (and often does) have infinitely many solutions when considering non-admissible partial low-weightings, and enumerating the set of admissible partial low-weightings is itself computationally complex. 

\subsection{Computational methods}\label{subsec: computation}

In this section, we describe the computational steps used to classify compact Coxeter polytopes of a given combinatorial type $G$, following the methods described in Proposition \ref{prop: iterate over partial} and Corollary \ref{cor: pass to computation}.  We begin with a set of $d + k$ vertices with unknown edge weights, and describe the process by which we assign edge weights such that every admissible partial low-weighting of type $G$ is constructed.

\subsubsection*{1. Select a set of multi-multiple edges}
For each combinatorial type, we begin by restricting the set of edges which can be multi-multiple according to the results in Subsection \ref{subsec: multi-multiple}.  In practice, each combinatorial type considered in this paper can be limited to having at most five multi-multiple edges.  We then iterate over all subsets of the possibly multi-multiple edges, fix such a subset $H$, and examine the Coxeter diagrams for which the set of multi-multiple edges is precisely $H$.

\subsubsection*{2. Assign edge weights within Lann{\'e}r diagrams of size $4$ and $5$}
Recall that there are finitely many Lann{\'e}r diagrams of size $4$ or $5$ (see Figure \ref{fig: Lanner diagrams}).  In particular, if we restrict an admissible partial low-weighting to only those vertices contained in Lann{\'e}r diagrams of size $4$ or $5$, the resulting subdiagram is obtained by gluing Lann{\'e}r diagrams from this list.  We can thus iterate over the possible Lann{\'e}r diagrams and over permutations of the vertices within each Lann{\'e}r diagram to assign weightings to edges within such a subdiagram.

\subsubsection*{3. Assign weights to the remaining ordinary low-weight edges}
There are finitely many assignments of the remaining low-weight edges, i.e., those not contained within a Lann{\'e}r diagram of size $4$ or $5$, since each must have weight $0$, $1$, $2$, or $3$.  We iterate over all these possibilities, with the additional restriction that all subdiagrams not containing a missing face must be elliptic (see Figure \ref{fig: ell diagrams}).  We furthermore require that subdiagram induced by any two Lann{\'e}r diagrams is connected, lest this subdiagram be superhyperbolic.  

\subsubsection*{4. Solve for the remaining edge weights}
At this stage, the only edge weights which have not been assigned are the weights of the multi-multiple edges and dashed edges.  The weights of the multi-multiple edges must be real numbers in the range %spacing
\\$\left[\cos\left(\frac{\pi}{6}\right), 1\right) = \left[\frac{\sqrt{3}}{2}, 1\right)$, and the weights of the dashed edges must be real numbers in the range $(1,\infty)$.  Restricting to these ranges, we find all solutions to the system of equations $V(G)$, where the unknowns are the weights of the multi-multiple and dashed edges.   We do so using the computer algebra system {\tt Mathematica}.  By assumption, $V(G)$ is chosen so that this solution set is finite.  

\subsubsection*{5. Check the signatures of the resulting matrices.}
For each of the solutions to the system of equations, we obtain a potential Gram matrix of a compact Coxeter polytope.  By Remark \ref{rem: Vinberg conditions}, it is sufficient to check whether the signature of the resulting matrix is $(d, 1, k-1)$.  If so, this is the Gram matrix of a compact Coxeter $d$-polytope with $d+k$ facets of combinatorial type $G$.  

\subsubsection*{6. Glue prisms onto any prism facets}
Recall that we initially assumed the weights of all prism edges were $0$.  If any polytopes are constructed by the process above, we can obtain the complete list of compact Coxeter polytopes of combinatorial type $G$ by gluing compact Coxeter prisms to this prism facet.  There are finitely many compact Coxeter prisms of dimension $4$ or $5$, see \cite{Kap} for a complete list.

The code implementing the process described above is publicly available at \\\url{https://github.com/agburcroff/Cox_d-Polytopes_with_dplus4_Facets.git}.

\section{\texorpdfstring{Classification of Compact Coxeter $4$-Polytopes with $8$ Facets}{Classification of Compact Coxeter 4-Polytopes with 8 Facets}}\label{sec: d+4 class}
In this section, we explain how to determine all compact Coxeter $4$-polytopes with $8$ facets.  We handle each of the $34$ possible combinatorial types determined in Section \ref{sec: comb types d+4} and listed in Appendix \ref{app: Gale diagrams} individually.  

For each such combinatorial type, we prove a sufficient number of restrictions on the Coxeter diagram of any polytope realising this type so that the possible Coxeter diagrams can be listed by a finite algorithm.  To be more specific, we claim that given any possible assignment of weights on the ordinary edges of $\Sigma$, there are a finite number of solutions for the weights of the remaining edges given the system of equations for $\rank(M(\Sigma))$ i.e, that every principal minor of order $6$ in $M(\Sigma)$ has determinant $0$.

Given a combinatorial type, the task of enumerating all compact polytopes with this type is a rather involved task.  Due to the presence of multi-multiple edges, initially there is an infinite number of potential Gram matrices that must be searched.  In this section, we detail how to use information about the elliptic and Lann{\'e}r subdiagrams to reduce the number of multi-multiple edges, then, once we have reduced to a system of algebraic equations with finitely many solutions, computationally determine those solutions which yield compact polytopes.

We first eliminate several combinatorial types by applying the previous work of Felikson and Tumarkin \cite{FT1,FT2,FT3, Tum}.

\begin{cor}
There are no compact hyperbolic Coxeter $4$-polytopes of combinatorial types $G_{31}$, $G_{32}$, $G_{33}$, or $G_{34}$.
\end{cor}
\begin{proof}
These combinatorial types all have exactly one pair of disjoint facets, and thus cannot be realised by Theorem \ref{thm: FT one pair}.
\end{proof}

Using Corollary \ref{cor: 4 8 disjoint} (following from results in \cite{FT3}) we can immediately exclude six more combinatorial types, each having precisely two pairs of disjoint facets.  

\begin{cor}
There are no compact hyperbolic Coxeter $4$-polytopes of combinatorial types $G_{25}$, $G_{26}$, $G_{27}$, $G_{28}$, $G_{29}$, or $G_{30}$.
\end{cor}

We can furthermore eliminate three combinatorial types by considering a certain forbidden subdiagram.

\begin{cor}
There are no compact hyperbolic Coxeter $4$-polytopes of combinatorial types $G_{22}$, $G_{23}$, or $G_{24}.$
\end{cor}
\begin{proof}
Each of these has a subdiagram with induced missing face list isomorphic to $\{0123, 014, 235\}$, hence cannot be realised as a polytope by Lemma \ref{lem: Tum forbidden subdiagram}.
\end{proof}

It remains to determine whether there are polytopes of combinatorial type $G_i$ for $i = 1,2,\dots,21$, and if so, to classify these polytopes.  We handle each of these combinatorial types individually in the following subsections.  In each subsection, we prove a sufficient set of restrictions to ensure that any polytopes realising a given combinatorial type can be listed by a finite algorithm.  One such initial restriction is that we can set the weight of all prism edges to $0$ (see Section \ref{sec: d+4 properties} for further explanation), and then glue prisms to any polytopes obtained.  In each case, the algorithm has been successfully implemented via the program described in Subsection \ref{subsec: computation}, though this program can in some cases take several days to check one combinatorial type on a standard machine.

\addtocontents{toc}{\protect\setcounter{tocdepth}{1}}
\subsection*{\texorpdfstring{Combinatorial type $G_1$}{Combinatorial type G1}}
Observe that this is the combinatorial type of a simplex which has been truncated at three distinct vertices.  In particular, vertices $v_0$, $v_1$, and $v_2$, which correspond to the facets obtained by the truncation, are all prism vertices. There are no possible multi-multiple edges by Lemma \ref{lem: multi-multiple}, since all non-prism edges are contained in a Lann{\'e}r diagram of order $4$. This combinatorial type satisfies the conditions of Corollary \ref{cor: pass to computation} and can be handled via computer computation without further restriction using the following system of equations:
$$V(G_1) = \{01,02,04,12,13,23\}\,.$$
This yields $130$ compact Coxeter polytopes of type $G_1$.

\subsection*{\texorpdfstring{Combinatorial type $G_2$}{Combinatorial type G2}}
This is the combinatorial type of a simplex that has been truncated at two vertices and one edge, disjoint from the two vertices.  The vertices $v_0$ and $v_1$ are both prism vertices.  There are no possible multi-multiple edges, as all non-prism edges satisfy the hypotheses of Lemma \ref{lem: multi-multiple}.  Similarly to the previous case, this combinatorial type satisfies the conditions of Corollary \ref{cor: pass to computation} with:
$$V(G_2) = \{01, 03, 04, 12, 13, 23\}\,.$$
This yields $115$ compact Coxeter polytopes of type $G_2$.

\subsection*{\texorpdfstring{Combinatorial type $G_3$}{Combinatorial type G3}}
This combinatorial type, like the previous two, is also relatively straightforward to classify.  Note that the prism vertices are $v_0$ and $v_3$. There are no possible multi-multiple edges by Lemma \ref{lem: multi-multiple}. The system of equations 
$$V(G_3) = \{02,03,07,12,13,23\}$$
provides sufficient restrictions to satisfy Corollary \ref{cor: pass to computation}.  This yields $49$ compact Coxeter polytopes of type $G_3$.

\subsection*{\texorpdfstring{Combinatorial type $G_4$}{Combinatorial type G4}}
Since this combinatorial type consists of four pairs of disjoint facets, it is the combinatorial type of a $4$-cube.  The compact $4$-cubes were recently classified by Jacquemet and Tschantz \cite{JT}, which, similarly to the present methods, utilised significant computer searches.  There are $12$ such polytopes.

\subsection*{\texorpdfstring{Combinatorial type $G_5$}{Combinatorial type G5}}
This is the first example of a combinatorial type for which we need to apply extra restrictions before passing to computer computation.  

Note that there is at most one multi-multiple edge within the subdiagram $567$, and one multi-multiple edge within the subdiagram $01234$.  There are no further possible multi-multiple edges.  By symmetry, we can assume the multi-multiple edges are limited to $57$ and $23$.

Suppose that both arise, say $57$ and $23$ are multi-multiple.  Note then that $46$, $06$, and $16$ are non-empty to ensure no two Lann{\'e}r subdiagrams are disjoint, and in fact they must be single edges.  Moreover, $56$ and $57$ are either both have multiplicity $1$, or $0$ and $1$, respectively. The only other additional non-empty edges are $12$ and $03$, each with multiplicity at most $3$.  Up to symmetry, this leaves $2(1+3+9)=24$ diagrams to check.  We can do so using Corollary \ref{cor: pass to computation} and the system of equations
$$V(G_5) = \{03, 04, 12, 13, 14, 24, 34\}\,.$$

Now suppose just $57$ is multi-multiple.  The previous check actually encompasses the case where precisely three vertices of $01234$ are adjacent to $v_6$.  Suppose there are four such vertices, say $0123$ are.  Then all edges between the subdiagram $0123$ and $v_6$ are simple edges, and $34$ and $45$ either both have multiplicity $1$, or $0$ and $1$, respectively. Then the only additional edges are $02$ and $27$, each with multiplicity at most $3$.  As in the previous case, there are $24$ partial low-weightings to test, we can complete this using Corollary \ref{cor: pass to computation} and the system of equations $V(G_5)$ defined above. Lastly, suppose there are five such vertices, then we can proceed similarly with just $2$ diagrams to check.

The remaining cases, namely when only $23$ is multi-multiple or when there are no multi-multiple edges, can be handled with the same choice of $V(G_5)$ as above.  This yields $3$ compact Coxeter polytopes of type $G_5$.  

\subsection*{\texorpdfstring{Combinatorial type $G_6$}{Combinatorial type G6}}

Note that $15$ and $35$ are not multi-multiple by applying Lemma \ref{lem: overlapping triangles} to the subdiagram $13457$.  

If $36$ is multi-multiple, then so is $37$.  Suppose it is, and $37$ is not.  Then $46$, $56$, and $35$ are empty to prevent inducing a forbidden missing face of size $3$.  Then $57$ has multiplicity at least $2$ - looking at the Lann{\'e}r subdiagrams, there are only two possibilities for the rank-$4$ Lann{\'e}r diagram (where in fact $57$ has multiplicity $3$).  It can be checked that the left triangle $145$ cannot be admissibly completed in either case. 

Suppose $36$ and $37$ are multi-multiple.  Then $4567$ is a Lann{\'e}r path.  We can then use local determinants on $134567$, along with Proposition \ref{prop: local edge}, to show that there are no possibilities.

Otherwise, just $37$, $12$, $24$, and $14$ can be multi-multiple.  Under this assumption, $G_6$ satisfies the hypotheses of Corollary \ref{cor: pass to computation} with 
$$V(G_6) = \{01,02,06,12,14,16,23,26,27\}\,.$$
This yields $2$ compact Coxeter polytopes of type $G_6$.

\subsection*{\texorpdfstring{Combinatorial type $G_7$}{Combinatorial type G7}}
Note that vertices $v_0$ and $v_1$ are both prism vertices.  The only possible multi-multiple edges are $34$ and one of $25$ or $26$ by Lemma \ref{lem: multi-multiple}.  Under these restrictions, we can apply Corollary \ref{cor: pass to computation} with 
$$V(G_7) = \{02, 03, 07, 12, 13, 23\}\,.$$
This yields $2$ compact Coxeter polytopes of type $G_7$.

\subsection*{\texorpdfstring{Combinatorial type $G_8$}{Combinatorial type G8}}
This has prism vertices $v_0$ and $v_3$, with the only possible multi-multiple edges being $16$ and $27$ by Lemma \ref{lem: multi-multiple}.   The system of equations 
$$V(G_8) = \{03, 05, 06, 07, 13, 23, 67\}$$
provides sufficient restrictions to satisfy Corollary \ref{cor: pass to computation}.  This yields $1$ compact Coxeter polytope of type $G_8$.

\subsection*{\texorpdfstring{Combinatorial type $G_9$}{Combinatorial type G9}}
Note that the vertex $v_0$ is the only prism vertex.    By Lemma \ref{lem: multi-multiple}, the only possible multi-multiple edges are $36$, $37$, and $15$.  Note that $36$ and $37$ cannot simultaneously be multi-multiple, otherwise any assignment of weights to the rank-$4$ Lann{\'e}r diagram $4567$ induces a forbidden Lann{\'e}r triangle with $3$. By symmetry, we can assume $37$ is not multi-multiple.

Suppose $36$ and $15$ are multi-multiple.  Note that $46$, $56$, and $57$ must then be empty, hence the rank-$4$ Lann{\'e}r diagram $4567$ is a path.  For any assignment of weights within this Lann{\'e}r diagram, there is no additional non-empty edge connecting Lann{\'e}r diagram $25$ to the Lann{\'e}r diagram $367$.  Thus, this cannot occur, as we cannot have two disjoint Lann{\'e}r diagrams.

Suppose only edge $15$ is multi-multiple.  The hypotheses of Corollary \ref{cor: pass to computation} are satisfied with
$$V(G_9) = \{01, 04, 05, 12, 13, 15, 45\}\,.$$

Suppose only edge $36$ is multi-multiple.  We can again look at the structure of the rank-$4$ Lann{\'e}r diagram (noting $46$ and $56$ are empty) - if it is a path, then we can repeat the same analysis as two paragraphs higher with connecting $25$ to $367$.  The only remaining option is that the rank-$4$ diagram consists of the star with one edge of multiplicity $3$ and two edges of multiplicity $1$, with two possible positions (up to symmetry).  Note that the only additional edges are $15$ and $24$.  We complete the analysis of this case with $V(G_9)$ as above.

Similarly, if all ordinary edges have low weight, then the same system of equations $V(G_9)$ suffices.  This yields $15$ compact Coxeter polytopes of type $G_9$.

\subsection*{\texorpdfstring{Combinatorial type $G_{10}$}{Combinatorial type G10}}
By Lemma \ref{lem: overlapping triangles}, the only possible multi-multiple edges are in the subdiagrams $256$ and $034$.  Note that not both $25$ and $26$ can be multi-multiple lest $156$ cannot be admissibly completed, so assume $26$ has low weight.  A similar analysis can be used to show that at least one of $03$ or $04$ has low weight, hence we can assume $04$ has low weight.

Suppose both $34$ and $56$ are multi-multiple.  Note that we cannot have all edges $13$, $14$, $15$, $16$ non-empty, less we form a $4$-cycle in the subdiagram $134567$.  This subdiagram cannot contain a $4$-cycle since it contains no Lann{\'e}r diagrams of size at least $4$, and all elliptic diagrams are acyclic.  Thus, we can assume by symmetry that the edge $13$ is empty.  Applying Proposition \ref{prop: local edge} to the subdiagram $013456$, we obtain that at least one of $56$ or $34$ has multiplicity at most $13$.  Under this assumption, we can complete the analysis via Corollary \ref{cor: pass to computation} with
$$V(G_{10}) = \{02, 03, 05, 06, 12, 13, 23\}\,.$$

Otherwise, assuming the multi-multiple edges are limited to $03$, $25$, and at most one of $34$ or $56$, Corollary \ref{cor: pass to computation} similarly applies with $V(G_{10})$ defined as above.  This yields $1$ compact Coxeter polytope of type $G_{10}$.

\subsection*{\texorpdfstring{Combinatorial type $G_{11}$}{Combinatorial type G11}}
Observe that the only possible multi-multiple edges are those within the subdiagrams $0167$ or $2345$.  

Suppose that $45$ is multi-multiple.  Observe that no other edges of $2345$ are multi-multiple, lest $67$ cannot connect to one of the rank-$3$ Lann{\'e}r diagrams.  Moreover, by symmetry we can assume the multi-multiple edges among $0167$ are limited to $06$ and $17$ or $06$ and $16$.  Suppose first that these two edges are $06$ and $17$.  We can apply Corollary \ref{cor: pass to computation} with 
$$V(G_{11}) = \{01, 02, 03, 04, 05, 06, 16, 17, 27, 67\}\,.$$
The same analysis can be applied supposing that these two edges are $06$ and $16$.  Thus, we can now assume $45$ is not multi-multiple.

Suppose $06$ is multi-multiple.  In order to ensure that $67$ is connected to $245$ and $345$, we have three cases:
\begin{enumerate}
    \item Suppose $26$ and $37$ are non-empty.  The only other possible multi-multiple edge is $17$.
    \item Suppose $27$ and $37$ are non-empty.  The only other possible multi-multiple edge is $27$.
    \item Suppose $57$ is non-empty.   The other possible multi-multiple edges are $24$, $34$, and $16$.  We cannot have $24$ and $34$ simultaneously multi-multiple, lest we create a forbidden Lann{\'e}r triangle, so assume by symmetry that $24$ is not multi-multiple. 
\end{enumerate}
In all three cases, we can apply Corollary \ref{cor: pass to computation} with $V(G_{11})$ as above.  We can now assume the multi-multiple edges are among $1345$ excluding $34$, i.e. either $13$, $45$ or $35$, $45$ or just one of these, and again the same analysis applies.  This yields $8$ compact Coxeter polytopes of type $G_{11}$.

\subsection*{\texorpdfstring{Combinatorial type $G_{12}$}{Combinatorial type G12}}
Observe that $57$, $37$, $47$, $67$, and $27$ are not multi-multiple by Lemma \ref{lem: overlapping triangles}.

Suppose $03$ is multi-multiple, and $34$ is not.  Then $47$ is empty, and $34$ and $37$ have multiplicities $2,3$ or $3,3$, respectively.  Thus, $67$ has multiplicity $1$.  Considering local determinants on $013457$, the only two possibilities are that $03$ has multiplicity $4$, $34$ has multiplicity $2$ or $3$, $37$ has multiplicity $3$, and $56$ has multiplicity $5$.  Note that now $03$ and $13$ must be empty.  If edge $27$ is empty, we can consider local determinants on $234567$ using Proposition \ref{prop: local edge} to see that we must have multiplicity $2$ on $34$, and repeating local determinant analysis on $123467$ we find that the multiplicity of $26$ is $5$, $6$, or $7$.  Then considering local determinants on the subdiagram $234567$ yieldsthat $17$ is empty, and no other can arise.  This leaves us with three diagrams to check - none of these yield polytopes.  Now we can assume that $27$ is non-empty - in fact, it must have multiplicity $1$.  Considering local determinants on the subdiagram $123467$ using Proposition \ref{prop: local edge} yields that the respective weights of $45$, $34$, $67$ are $2,3,2$; $2,3,3$; $3,3,2$; or $3,2,2$.  Moreover, the only other possible non-empty edges are $25$ or $12$, each with multiplicity $1$.  This leaves us with $16$ diagrams to check - none of these yield polytopes.

Suppose $03$ and $34$ are multi-multiple.  Then $07$ is empty.  So we can consider local determinants on $034567$, noting that if $37$ has weight larger than $1$ then $67$ has multiplicity $1$.  From this and the previous paragraph, we can determine that there are no admissible diagrams with $03$ multi-multiple.  

We are now left with considering multi-multiple edges $34$, $56$, $26$, $25$.  We can then apply Corollary \ref{cor: pass to computation} with 
$$V(G_{12}) = \{05, 12, 15, 16, 25, 26, 35, 56\}\,.$$
This yields $4$ compact Coxeter polytopes of type $G_{12}$.

\subsection*{\texorpdfstring{Combinatorial type $G_{13}$}{Combinatorial type G13}}
Observe by Lemmas \ref{lem: multi-multiple} and \ref{lem: overlapping triangles} that the only possible multi-multiple edges are $12$, $56$, $16$, $25$, $34$, $03$, and $04$.  We cannot simultaneously have $03$ and $04$ multi-multiple lest $37$ and $47$ are both empty, so assume $04$ is not.   We cannot simultaneously have $16$ and $25$ multi-multiple lest $57$ and $47$ are both empty, so assume $25$ is not.  Thus, we are limited to $12$, $56$, $16$, $34$, and $03$ being multi-multiple.

Suppose $56$ and $03$ are multi-multiple.  Then $37$ and $03$ are empty, so we can consider local determinants on $034567$ using Proposition \ref{prop: local edge}.  Assume $57$ is non-empty.  Note that either $67$ is non-empty, or $2$ is attached to $347$ to ensure $26$ is attached to $347$.  If $67$ is non-empty, the local determinant on $034567$ with respect to $07$ cannot be zero, which is a contradiction.  If $2$ is attached to $7$ or $4$, then $47$ cannot have multiplicity $3$ so is $2$.  Considering the same local determinant shows that this is not possible.  So $23$ is non-empty, and by local determinants we must have that $03$ has multiplicity $4$ or $5$, $34$ has multiplicity $2$ or $3$, $47$ has multiplicity $3$, and both $57$ and $23$ have multiplicity $1$, with no other edges being nonempty.  Now take the determinant of the subdiagram $034567$; this is non-zero, a contradiction.

Suppose $34$ and $03$ are multi-multiple (we can assume $56$ is not from previous paragraph).  Then $37$ is empty, so $47$ is not.  We can now use local determinants on $034567$.  From this, we determine that $567$ must have one empty edge.  Thus, one of $67$, $57$ has multiplicity at least $2$, so $47$ has multiplicity $1$.  Repeating a similar local determinant argument, we find that there are no admissible diagrams of this type.

The remaining cases can be handled by Corollary \ref{cor: pass to computation} with 
$$V(G_{13}) = \{01, 02, 05, 06, 12, 13, 16, 25, 26\}\,.$$
This yields $4$ compact Coxeter polytopes of type $G_{13}$.

\subsection*{\texorpdfstring{Combinatorial type $G_{14}$}{Combinatorial type G14}}
By similar considerations to type $G_{12}$, if any of $27$, $37$, $26$, $36$, $02$, or $03$ is multi-multiple then so is $23$.  

Suppose $27$ and $23$ are multi-multiple.  Note that the other possible edges (up to symmetry) can be limited to within $145$, $56$, $04$, $03$, and $36$.  
\begin{itemize}
    \item Suppose $56$ and $04$ are non-empty.  Then the only other possible multi-multiple edge is $14$.  
    \item Suppose $56$ is non-empty, $04$ is empty.  Then the only other possible multi-multiple edges are $45$ and $14$.  
    \item Suppose $04$ is non-empty, $56$ is empty.  Then the only other possible multi-multiple edges are $15$ and $14$. 
    \item Suppose $56$ and $04$ are empty.  Then the other possible multi-multiple edges are $45$, $14$, and $15$.  The only additional edges are $03$ and $36$, one of which has multiplicity $1$ or $2$ and the other having multiplicity $1$.  
\end{itemize}

In all these cases, we can apply Corollary \ref{cor: pass to computation} with 
$$V(G_{14}) = \{01, 14, 15, 17, 23, 24, 47, 57\}\,.$$

We can now restrict the multi-multiple edges to $23$, $45$, $14$, and $15$, and the same analysis applies along with the same choice of $V(G_{14})$.  This yields $2$ compact Coxeter polytopes of type $G_{14}$.

\subsection*{\texorpdfstring{Combinatorial type $G_{15}$}{Combinatorial type G15}}
Look at the subdiagram generated by $234567$.  It can quickly be checked that the only possible multi-multiple edge is one of $26$ or $27$, suppose by symmetry it is $26$.  Then $36$ must be empty, so $37$ and $67$ must have multiplicities $2$ and $3$ or both $3$.  From this information, we can then complete the remainder of the rank-$4$ Lann{\'e}r diagram, keeping in mind that $46$ must have multiplicity $0$ as well.  After doing so, it can be easily checked that the triangle $345$ cannot be completed to an admissible diagram.  Hence there are no compact Coxeter polytopes of this type.

\subsection*{\texorpdfstring{Combinatorial type $G_{16}$}{Combinatorial type G16}}
Considering overlapping triangles, the only possible multi-multiple edges are $15$, $12$, $03$, $04$, and $34$.  We cannot have $03$ and $04$ simultaneously multi-multiple, so assume $03$ has low weight.  By similar considerations to $G_{12}$, if $04$ is multi-multiple then so is $34$.  

We can apply Corollary \ref{cor: pass to computation} with 
$$V(G_{16}) = \{01, 02, 07, 12, 17, 24, 25\}$$
to show that there are no compact Coxeter polytopes of this type.

\subsection*{\texorpdfstring{Combinatorial type $G_{17}$}{Combinatorial type G17}}
By Lemma \ref{lem: overlapping triangles}, the only possible multi-multiple edges are $36$, $37$, $14$, and $15$.  Note that $36$ and $37$ cannot be simultaneously multi-multiple; assume by symmetry that $36$ has low weight.   Moreover, $14$ and $15$ cannot simultaneously be multi-multiple.

With these assumptions, we can apply Corollary \ref{cor: pass to computation} with 
$$V(G_{17}) = \{03, 05, 07, 13, 34, 35, 57\}$$
to show that there are no compact Coxeter polytopes of this type.

\subsection*{\texorpdfstring{Combinatorial type $G_{18}$}{Combinatorial type G18}}
Considering overlapping triangles, the only possible multi-multiple edges are $12$, $37$, $36$, $25$, and $14$.  

Suppose all five of these edges are multi-multiple.  Then the rank-$4$ Lann{\'e}r diagram must be a path, and there are no additional non-empty edges.  We can check that there are no solutions to the determinants of the subdiagrams $123456$ and $012345$ being zero, assuming that the undetermined edges have weight in the range $\left[-\cos\left(\frac{\pi}{6}\right), 0\right)$.  Thus, we can assume at most four of these edges are multi-multiple.  We can now complete the analysis with Corollary \ref{cor: pass to computation}, setting 
$$V(G_{18}) = \{01, 02, 03, 04, 05, 12, 23, 13\}\,,$$
to show that there are no compact Coxeter polytopes of this type.

\subsection*{\texorpdfstring{Combinatorial type $G_{19}$}{Combinatorial type G19}}
By Lemma \ref{lem: overlapping triangles}, the only possible multi-multiple edges are $15$, $14$, $35$, $37$, $24$, and $26$.  Note $35$ and $15$ cannot simultaneously be multi-multiple, and similarly for $24$ and $14$.  We can now complete the analysis with Corollary \ref{cor: pass to computation}, setting 
$$V(G_{19}) = \{02, 03, 05, 06, 07, 12, 13, 23\}\,,$$
to show that there are no compact Coxeter polytopes of this type.

\subsection*{\texorpdfstring{Combinatorial type $G_{20}$}{Combinatorial type G20}}
By Lemma \ref{lem: overlapping triangles}, the only possible multi-multiple edges are contained in the subdiagram $167$.  We can then apply Corollary \ref{cor: pass to computation} with  
$$V(G_{20}) = \{05, 06, 12, 15, 16, 56\}\,.$$
This shows that there are no compact Coxeter polytopes of this type.

\subsection*{\texorpdfstring{Combinatorial type $G_{21}$}{Combinatorial type G21}}
By Corollary \ref{cor: pass to computation}, the only possible multi-multiple edges are $05$, $07$, $15$, $13$, $34$, $46$, $26$, and $27$.  Note that $05$ and $07$ cannot simultaneously be multi-multiple, lest $57$ cannot connect to $346$.  By symmetry, assume $07$ is not multi-multiple.  We can then apply Corollary \ref{cor: pass to computation} with  
$$V(G_{21}) = \{01, 02, 03, 04, 05, 06, 07, 12, 15, 25\}\,,$$
to show that there are no compact Coxeter polytopes of this type.

\section{\texorpdfstring{Classification of Compact Coxeter $5$-Polytopes with $9$ Facets}{Classification of Compact Coxeter 5-Polytopes with 9 Facets}}\label{sec: d+5 class}
We now proceed to the classification of compact Coxeter $5$-polytopes with $9$ facets.  Though the process detailed in Section \ref{sec: comb types d+4} yields more possible combinatorial types in dimension $5$ than in dimension $4$, only six of these ($H_i$ for $1 \leq i \leq 6$) are realised by compact Coxeter polytopes.  Moreover, these combinatorial types in general have a more restrictive face structure than in the lower-dimensional cases, thus in most cases find a set of equations for which the hypotheses of Corollary \ref{cor: pass to computation} are satisfied without first proving additional restrictions.  The main challenges are limiting the set of multi-multiple edges and determining the set of equations $V(H)$ for each combinatorial type $H$.  We examine this process in detail for a few combinatorial types, but we often omit the argument for routine cases.  In these cases, Lemmas \ref{lem: multi-multiple} and \ref{lem: overlapping triangles} are sufficient to limit the multi-multiple edges, and of course the claimed properties of $V(H)$ are checked computationally.

\vspace{-0.2cm}
\subsubsection*{\texorpdfstring{Combinatorial type $H_{1}$}{Combinatorial type H1}}  Note that $v_0$ and $v_1$ are both prism vertices.  By \cite[Lemma 5.3]{Ess}, it is easily checked that there are four possible subdiagrams with the proper missing face structure on $2345678$.  We can thus obtain all diagrams by taking these four subdiagrams and gluing prisms onto facets $0$ and $1$.  This process yields $22$ compact polytopes of type $H_1$.

\vspace{-0.2cm}
\subsubsection*{\texorpdfstring{Combinatorial type $H_{2}$}{Combinatorial type H2}} 
Note that $v_0$ and $v_3$ are both prism vertices.  It is straightforward to check that there are only four possible subdiagrams for $1245678$ by gluing together subdiagrams from \cite[Lemma 5.3]{Ess}.  We thus obtain all $18$ compact polytopes of type $H_2$ by gluing prisms as appropriate to these subdiagrams.

\vspace{-0.2cm}
\subsubsection*{\texorpdfstring{Combinatorial type $H_{3}$}{Combinatorial type H3}}  
At most one of the edges $25$ or $26$ can be multi-multiple.  By symmetry, we can assume that only $25$ can be multi-multiple.  We can then apply Corollary \ref{cor: pass to computation} with $V(H_3) = \{23, 15, 12, 04, 02, 01, 13, 15\}$.  This yields $6$ compact Coxeter polytopes of type $H_3$.

\vspace{-0.2cm}
\subsubsection*{\texorpdfstring{Combinatorial type $H_{4}$}{Combinatorial type H4}}
Only possible multi-multiple edges are $37$ or $38$.  These can not be multi-multiple simultaneously.  We can then take $V(H_4) = \{02, 03, 07, 08, 23, 27, 38\}$.  This yields $3$ compact Coxeter polytopes of type $H_4$.

\vspace{-0.2cm}
\subsubsection*{\texorpdfstring{Combinatorial type $H_{5}$}{Combinatorial type H5}}
The possible multi-multiple edges are restricted to $25$, $04$, $03$, $14$, and $13$, some of which cannot be multi-multiple simultaneously. By symmetry between $3$ and $4$, we can restrict to $25$, $04$, and $14$ possibly being multi-multiple.  We then take $V(H_5)= \{01, 03, 05, 12, 14, 45\}$.  This yields $1$ compact Coxeter polytope of type $H_5$.

\vspace{-0.2cm}
\subsubsection*{\texorpdfstring{Combinatorial type $H_{6}$}{Combinatorial type H6}}
There are no possible multi-multiple edges.  We can take
$V(H_6) = \{01, 02, 04, 12, 13, 23\}$.  This yields $1$ compact Coxeter polytope of type $H_6$.

\vspace{-0.2cm}
\subsubsection*{\texorpdfstring{Combinatorial type $H_{7}$}{Combinatorial type H7}}
At most one of the edges $27$ or $28$ can be multi-multiple.  By symmetry, we can assume that only $28$ can be multi-multiple.  We can take 
$V(H_7) = \{02, 03, 12, 13, 23\}$.

\vspace{-0.2cm}
\subsubsection*{\texorpdfstring{Combinatorial type $H_{8}$}{Combinatorial type H8}}  
The only possible multi-multiple edge is $28$.  We can take 
$V(H_8) = \{02, 03, 08, 13, 23, 27, 28, 38\}$.

\vspace{-0.2cm}
\subsubsection*{\texorpdfstring{Combinatorial type $H_{9}$}{Combinatorial type H9}}  
The only possible multi-multiple edge is $34$.  We can take 
$V(H_9) = \{01, 02, 03, 04, 12, 13, 14, 23\}$.

\vspace{-0.2cm}
\subsubsection*{\texorpdfstring{Combinatorial type $H_{10}$}{Combinatorial type H10}}  
The only possible multi-multiple edge is $12$.  We can take 
$V(H_{10}) = \{01, 02, 03, 12, 18, 28\}$.

\vspace{-0.2cm}
\subsubsection*{\texorpdfstring{Combinatorial type $H_{11}$}{Combinatorial type H11}}  
The only possible multi-multiple edges are  $36$ or $37$, but not both.  As these cannot be multi-multiple simultaneously, we can assume by symmetry that $37$ has low weight.   We can take 
$V(H_{11}) = \{02, 03, 07, 08, 12, 23, 28, 38\}$.

\vspace{-0.2cm}
\subsubsection*{\texorpdfstring{Combinatorial type $H_{12}$}{Combinatorial type H12}}  
The only possible multi-multiple edges either $04$ and $14$, or $03$ and $13$.  By symmetry, we can assume the possible multi-multiple edges are only $03$ and $13$.  We can take 
$V(H_{12}) = \{01, 02, 03, 05, 12, 13, 35\}$.

\vspace{-0.2cm}
\subsubsection*{\texorpdfstring{Combinatorial type $H_{13}$}{Combinatorial type H13}}
The only possible multi-multiple edges are one of $26$ or $27$, and one of $14$ or $15$.  By symmetry, we can assume the multi-multiple edges are limited to $14$ and $26$.  We can take 
$V(H_{13}) = \{01, 02, 04, 06, 12, 13, 23\}$.

\vspace{-0.2cm}
\subsubsection*{\texorpdfstring{Combinatorial type $H_{14}$}{Combinatorial type H14}}
The only possible multi-multiple edges are one of $14$ or $15$.  By symmetry, we can assume $15$ has low weight.  We can take 
$V(H_{14}) = \{02, 05, 12, 13, 23, 24\}$.

\vspace{-0.2cm}
\subsubsection*{\texorpdfstring{Combinatorial type $H_{15}$}{Combinatorial type H15}}
The only possible multi-multiple edges are one of $14$ or $15$. We can take 
$V(H_{15}) = \{01, 03, 04, 05, 12, 13, 15, 34, 35\}$.

\vspace{-0.2cm}
\subsubsection*{\texorpdfstring{Combinatorial type $H_{16}$}{Combinatorial type H16}}
The multi-multiple edges limited to one of $03$ or $04$, and one of $13$ or $15$.   We can take 
$V(H_{16}) = \{01, 03, 05, 06, 12, 13, 16, 36\}$.

\vspace{-0.2cm}
\subsubsection*{\texorpdfstring{Combinatorial type $H_{17}$}{Combinatorial type H17}}
The multi-multiple edges are limited to one of $03$ or $04$, and one of $27$ or $28$.  By symmetry can limit ourselves to considering only $03$ and $27$.  We can take 
$V(H_{17}) = \{02, 06, 07, 12, 26, 67\}$.

\vspace{-0.2cm}
\subsubsection*{\texorpdfstring{Combinatorial type $H_{18}$}{Combinatorial type H18}}
The multi-multiple edges are limited to $12$ and $14$.  We can take 
$V(H_{18}) = \{02, 04, 12, 13, 23\}$.

\vspace{-0.2cm}
\subsubsection*{\texorpdfstring{Combinatorial type $H_{19}$}{Combinatorial type H19}}
The multi-multiple edges are limited to $27$, $28$, $37$, and $38$.  We can take 
$V(H_{19}) = \{02, 03, 07, 08, 12, 13, 23\}$.

\vspace{-0.2cm}
\subsubsection*{\texorpdfstring{Combinatorial type $H_{20}$}{Combinatorial type H20}}
The multi-multiple edges are limited to $14$, $17$, $34$, and $37$.  We can take 
$V(H_{20}) = \{01, 03, 04, 07, 12, 13, 23\}$.

\vspace{-0.2cm}
\subsubsection*{\texorpdfstring{Combinatorial type $H_{21}$}{Combinatorial type H21}}
The only possible multi-multiple edge are $37$ or $38$, and by symmetry we can assume $38$ has low weight.  We can then take 
$V(H_{21}) = \{06, 07, 13, 23, 67\}$.

\vspace{-0.2cm}
\subsubsection*{\texorpdfstring{Combinatorial type $H_{22}$}{Combinatorial type H22}}
The only possible multi-multiple edges are $37$ and $38$.  We can take 
$V(H_{22}) = \{07, 08, 12, 13, 23\}$.

\vspace{-0.2cm}
\subsubsection*{\texorpdfstring{Combinatorial type $H_{23}$}{Combinatorial type H23}}
The only possible multi-multiple edges are $16$ and one of $37$ or $38$.  By symmetry, we can assume $38$ has low weight.  We can take 
$V(H_{23}) = \{01, 03, 06, 12, 13, 23\}$.

\vspace{-0.2cm}
\subsubsection*{\texorpdfstring{Combinatorial type $H_{24}$}{Combinatorial type H24}}
The only possible multi-multiple edges are $24$ and $27$. We can take 
$V(H_{24}) = \{04, 07, 12, 13, 23\}$.

\vspace{-0.2cm}
\subsubsection*{\texorpdfstring{Combinatorial type $H_{25}$}{Combinatorial type H25}}
The only possible multi-multiple edges are $02$, $05$, $17$. We can take 
$V(H_{25}) = \{01, 02, 12, 14, 17, 23, 47\}$.

\vspace{-0.2cm}
\subsubsection*{\texorpdfstring{Combinatorial type $H_{26}$}{Combinatorial type H26}}
The only possible multi-multiple edges are $12$, $14$, $15$, and $24$. We can take 
$V(H_{26}) = \{01, 02, 04, 12, 14, 16, 23, 25\}$.

\vspace{-0.2cm}
\subsubsection*{\texorpdfstring{Combinatorial type $H_{27}$}{Combinatorial type H27}}
The only possible multi-multiple edges are $12$, $24$, and $25$. We can take 
$V(H_{27}) = \{04, 05, 12, 14, 23, 45\}$.

\vspace{-0.2cm}
\subsubsection*{\texorpdfstring{Combinatorial type $H_{28}$}{Combinatorial type H28}}
The only possible multi-multiple edges are $14$, $15$, and $24$. We can take 
$V(H_{28}) = \{01, 02, 04, 12, 13, 23\}$.

\vspace{-0.2cm}
\subsubsection*{\texorpdfstring{Combinatorial type $H_{29}$}{Combinatorial type H29}}
The only possible multi-multiple edges are $14$, $24$, and $27$. We can take 
$V(H_{29}) = \{01, 02, 04, 12, 13, 23\}$.

\vspace{-0.2cm}
\subsubsection*{\texorpdfstring{Combinatorial type $H_{30}$}{Combinatorial type H30}}
The only possible multi-multiple edge is $14$. We can take 
$V(H_{30}) = \{05, 12, 13, 23\}$.

\vspace{-0.2cm}
\subsubsection*{\texorpdfstring{Combinatorial type $H_{31}$}{Combinatorial type H31}}
The only possible multi-multiple edge is $38$. We can take 
$V(H_{31}) = \{07, 12, 13, 23\}$.

\vspace{-0.2cm}
\subsubsection*{\texorpdfstring{Combinatorial type $H_{32}$}{Combinatorial type H32}}
There are no possible multi-multiple edges. We can take 
$V(H_{32}) = \{12, 13, 23\}$.

\vspace{-0.2cm}
\subsubsection*{\texorpdfstring{Combinatorial type $H_{33}$}{Combinatorial type H33}}
The only possible multi-multiple edges are $15$ or $24$; by symmetry we can assume $24$ has low weight.  We can take 
$V(H_{33}) = \{01, 02, 05, 12, 13, 23\}$.

\vspace{-0.2cm}
\subsubsection*{\texorpdfstring{Combinatorial type $H_{34}$}{Combinatorial type H34}}
The only possible multi-multiple edges are one of $25$ or $26$, and one of $37$ or $38$. By symmetry, we can assume $26$ and $38$ have low weight.  We can take 
$V(H_{34}) = \{02, 05, 07, 12, 13, 23\}$.

\vspace{-0.2cm}
\subsubsection*{\texorpdfstring{Combinatorial type $H_{35}$}{Combinatorial type H35}}
The only possible multi-multiple edges are $25$, $26$, and one of $03$ or $04$.  By symmetry, we can assume $04$ has low weight.  We can take 
$V(H_{35}) = \{02, 05, 06, 12, 25, 56\}$.

\vspace{-0.2cm}
\subsubsection*{\texorpdfstring{Combinatorial type $H_{36}$}{Combinatorial type H36}}
The only possible multi-multiple edges are $34$ and one of $03$ or $04$.  By symmetry, we can assume $04$ has low weight.  We can take 
$V(H_{36}) = \{07, 12, 13, 23\}$.

\vspace{-0.2cm}
\subsubsection*{\texorpdfstring{Combinatorial type $H_{37}$}{Combinatorial type H37}}
The only possible multi-multiple edges are $27$ and one of $03$ or $04$.  By symmetry, we can assume $04$ has low weight.  We can take 
$V(H_{37}) = \{08, 12, 13, 23\}$.

\vspace{-0.2cm}
\subsubsection*{\texorpdfstring{Combinatorial type $H_{38}$}{Combinatorial type H38}}
The only possible multi-multiple edges are $07$, $08$, $17$, $18$.  We can take 
$V(H_{38}) = \{01, 03, 07, 08, 12, 17, 18, 27\}$.

\vspace{-0.2cm}
\subsubsection*{\texorpdfstring{Combinatorial type $H_{39}$}{Combinatorial type H39}}
The only possible multi-multiple edges are $34$ and one of $03$ or $04$.  By symmetry, we can assume $04$ has low weight.  We can take 
$V(H_{39}) = \{08, 12, 13, 23\}$.

\vspace{-0.2cm}
\subsubsection*{\texorpdfstring{Combinatorial type $H_{40}$}{Combinatorial type H40}}
The only possible multi-multiple edges are $15$ and one of $03$ or $04$.  By symmetry, we can assume $04$ has low weight.  We can take 
$V(H_{40}) = \{06, 12, 13, 23\}$.

\vspace{-0.2cm}
\subsubsection*{\texorpdfstring{Combinatorial type $H_{41}$}{Combinatorial type H41}}
The only possible multi-multiple edges are $03$, $04$, and $15$.  We can take 
$V(H_{41}) = \{06, 12, 13, 14, 23\}$.

\vspace{-0.2cm}
\subsubsection*{\texorpdfstring{Combinatorial type $H_{42}$}{Combinatorial type H42}}
The only possible multi-multiple edges are $02$, $03$, and $23$.  We can take 
$V(H_{42}) = \{03, 07, 12, 13, 23\}$.

\vspace{-0.2cm}
\subsubsection*{\texorpdfstring{Combinatorial type $H_{43}$}{Combinatorial type H43}}
The only possible multi-multiple edges are $03$ and $34$.  We can take 
$V(H_{43}) = \{01, 04, 13, 14, 23\}$.

\vspace{-0.2cm}
\subsubsection*{\texorpdfstring{Combinatorial type $H_{44}$}{Combinatorial type H44}}
The only possible multi-multiple edges are $12$, $15$, $25$, and one of $03$ or $04$.  By symmetry, we can assume $04$ has low weight. We can take 
$V(H_{44}) = \{02, 05, 07, 12, 17, 23, 57\}$.

\vspace{-0.2cm}
\subsubsection*{\texorpdfstring{Combinatorial type $H_{45}$}{Combinatorial type H45}}
The only possible multi-multiple edges are $03$, $04$, $13$, $15$, and $34$. We can take 
$V(H_{45}) = \{01, 03, 05, 12, 14, 23, 24\}$.

\vspace{-0.2cm}
\subsubsection*{\texorpdfstring{Combinatorial type $H_{46}$}{Combinatorial type H46}}
The only possible multi-multiple edges are $04$, $15$, and one of $03$ or $13$. By symmetry, we can assume $13$ has low weight.  We can take 
$V(H_{46}) = \{05, 06, 12, 13, 14, 24\}$.

\vspace{-0.2cm}
\subsubsection*{\texorpdfstring{Combinatorial type $H_{47}$}{Combinatorial type H47}}
The only possible multi-multiple edges are $03$, $04$, $13$, $15$, and $34$. We can take 
$V(H_{47}) = \{01, 03, 04, 12, 14, 23, 24\}$.

\vspace{-0.2cm}
\subsubsection*{\texorpdfstring{Combinatorial type $H_{48}$}{Combinatorial type H48}}
The only possible multi-multiple edges are $03$, $04$, and $34$. We can take 
$V(H_{48}) = \{05, 12, 13, 14, 23\}$.

\vspace{-0.2cm}
\subsubsection*{\texorpdfstring{Combinatorial type $H_{49}$}{Combinatorial type H49}}
The only possible multi-multiple edges are $27$, $28$, and $78$. We can take 
$V(H_{49}) = \{07, 08, 12, 27, 28, 78\}$.

\vspace{-0.2cm}
\subsubsection*{\texorpdfstring{Combinatorial type $H_{50}$}{Combinatorial type H50}}
The only possible multi-multiple edges are $34$ and one of $03$ or $04$.  By symmetry, we can assume $04$ has low weight.  We can take 
$V(H_{50}) = \{02, 04, 12, 13, 23\}$.

\vspace{-0.2cm}
\subsubsection*{\texorpdfstring{Combinatorial type $H_{51}$}{Combinatorial type H51}}
The only possible multi-multiple edges are $26$, $27$, and $67$.  We can take 
$V(H_{51}) = \{06, 07, 12, 26, 27, 67\}$.

\vspace{-0.2cm}
\subsubsection*{\texorpdfstring{Combinatorial type $H_{52}$}{Combinatorial type H52}}
The only possible multi-multiple edges are $27$, $28$, $67$.  We can take 
$V(H_{52}) = \{02, 07, 08, 27, 78\}$.

\vspace{-0.2cm}
\subsubsection*{\texorpdfstring{Combinatorial type $H_{53}$}{Combinatorial type H53}}
The only possible multi-multiple edges are $17$, $18$, $28$, and $78$.  We can take 
$V(H_{53}) = \{07, 08, 12, 17, 18, 78\}$.

\vspace{-0.2cm}
\subsubsection*{\texorpdfstring{Combinatorial type $H_{54}$}{Combinatorial type H54}}
The only possible multi-multiple edges are $16$, $27$, and $67$.  We can take 
$V(H_{54}) = \{02, 06, 07, 17, 26\}$.

\vspace{-0.2cm}
\subsubsection*{\texorpdfstring{Combinatorial type $H_{55}$}{Combinatorial type H55}}
The only possible multi-multiple edges are $34$ and one of $03$ or $04$.  By symmetry, we can assume $04$ has low weight.  We can take 
$V(H_{55}) = \{08, 12, 13, 23\}$.

\vspace{-0.2cm}
\subsubsection*{\texorpdfstring{Combinatorial type $H_{56}$}{Combinatorial type H56}}
The only possible multi-multiple edges are $13$, $16$, and $67$. We can take 
$V(H_{56}) = \{01, 03, 06, 16, 36\}$.

\vspace{-0.2cm}
\subsubsection*{\texorpdfstring{Combinatorial type $H_{57}$}{Combinatorial type H57}}
The only possible multi-multiple edges are $07$, $26$, $27$, and one of $03$ or $04$. By symmetry, we can assume $04$ has low weight.  We can take 
$V(H_{57}) = \{02, 06, 07, 17, 23, 27, 37\}$.

\vspace{-0.2cm}
\subsubsection*{\texorpdfstring{Combinatorial type $H_{58}$}{Combinatorial type H58}}
The only possible multi-multiple edges are $25$, $28$, and $68$.  We can take 
$V(H_{58}) = \{05, 08, 12, 28, 58\}$.

\vspace{-0.2cm}
\subsubsection*{\texorpdfstring{Combinatorial type $H_{59}$}{Combinatorial type H59}}
The only possible multi-multiple edges are $03$, $04$, $15$, $16$.  We can take 
$V(H_{59}) = \{01, 05, 06, 13, 14, 24, 34\}$.

\vspace{-0.2cm}
\subsubsection*{\texorpdfstring{Combinatorial type $H_{60}$}{Combinatorial type H60}}
The only possible multi-multiple edges are $02$, $03$, $13$, $15$.  We can take 
$V(H_{60}) = \{01, 03, 08, 13, 18, 23\}$.

\vspace{-0.2cm}
\subsubsection*{\texorpdfstring{Combinatorial type $H_{61}$}{Combinatorial type H61}}
The only possible multi-multiple edges are $12$, $15$, $26$, and one of $03$ or $04$. By symmetry, we can assume $04$ has low weight.  We can take 
$V(H_{61}) = \{01, 02, 06, 12, 13, 23\}$.

\vspace{-0.2cm}
\subsubsection*{\texorpdfstring{Combinatorial type $H_{62}$}{Combinatorial type H62}}
The only possible multi-multiple edges are $02$, $04$, and $24$.  We can take 
$V(H_{62}) = \{02, 07, 14, 24, 27\}$.

\vspace{-0.2cm}
\subsubsection*{\texorpdfstring{Combinatorial type $H_{63}$}{Combinatorial type H63}}
The only possible multi-multiple edges are $03$ and $04$.  We can take 
$V(H_{63}) = \{12, 13, 14, 24\}$.

\vspace{-0.2cm}
\subsubsection*{\texorpdfstring{Combinatorial type $H_{64}$}{Combinatorial type H64}}
The only possible multi-multiple edges are $67$ and one of $17$ or $26$.  By symmetry, we can assume $26$ has low weight.  We can take 
$V(H_{64}) = \{06, 07, 12, 16, 17, 67\}$.

\vspace{-0.2cm}
\subsubsection*{\texorpdfstring{Combinatorial type $H_{65}$}{Combinatorial type H65}}
The only possible multi-multiple edge is $67$.  We can take 
$V(H_{65}) = \{01, 17, 26\}$.

\vspace{-0.2cm}
\subsubsection*{\texorpdfstring{Combinatorial type $H_{66}$}{Combinatorial type H66}}
The only possible multi-multiple edges are $02$, $04$, $14$, $23$, and one of $15$ or $16$.  By symmetry, we can assume $16$ has low weight.  We can take 
$V(H_{66}) = \{01, 04, 08, 12, 14, 15, 24\}$.

\vspace{-0.2cm}
\subsubsection*{\texorpdfstring{Combinatorial type $H_{67}$}{Combinatorial type H67}}
The only possible multi-multiple edges are $13$ and $35$.  We can take 
$V(H_{67}) = \{02, 12, 13, 23\}$.

\vspace{-0.2cm}
\subsubsection*{\texorpdfstring{Combinatorial type $H_{68}$}{Combinatorial type H68}}
The only possible multi-multiple edges are $34$, $03$, $13$, $04$, and $14$. Note that if $03$ is multi-multiple, then $04$ and $14$ have low weight, lest a forbidden Lann{\'e}r diagram be induced among $01348$. Similar arguments hold when swapping the roles of vertices $3,4$ or $0,1$, so by symmetry we can assume $04$ and $14$ have low weight.  Then we can take 
$V(H_{68}) = \{01, 03, 05, 12, 13, 23\}$.

\vspace{-0.2cm}
\subsubsection*{\texorpdfstring{Combinatorial type $H_{69}$}{Combinatorial type H69}}
The only possible multi-multiple edges are $03$, $04$, $13$, and $15$.  We can take 
$V(H_{69}) = \{03, 05, 12, 13, 14, 23\}$.

\vspace{-0.2cm}
\subsubsection*{\texorpdfstring{Combinatorial type $H_{70}$}{Combinatorial type H70}}
The only possible multi-multiple edges are $03$, $04$, $15$, and $28$.  We can take 
$V(H_{70}) = \{05, 08, 12, 13, 23, 24\}$.

\vspace{-0.2cm}
\subsubsection*{\texorpdfstring{Combinatorial type $H_{71}$}{Combinatorial type H71}}
The only possible multi-multiple edges are $03$, $04$, and one of $12$ or $25$.  By symmetry, we can assume $25$ has low weight.  We can take 
$V(H_{71}) = \{01, 08, 12, 18, 23, 58\}$.

\vspace{-0.2cm}
\subsubsection*{\texorpdfstring{Combinatorial type $H_{72}$}{Combinatorial type H72}}
The only possible multi-multiple edges are $03$, $13$, and $15$.  We can take 
$V(H_{72}) = \{03, 05, 12, 13, 23\}$.

\vspace{-0.2cm}
\subsubsection*{\texorpdfstring{Combinatorial type $H_{73}$}{Combinatorial type H73}}
The possible multi-multiple edges are $03$, $04$, $13$, $15$, $67$.  We can take 
$V(H_{73}) = \{01, 03, 06, 12, 13, 14, 35, 36\}$.

\vspace{-0.2cm}
\subsubsection*{\texorpdfstring{Combinatorial type $H_{74}$}{Combinatorial type H74}}
The only possible multi-multiple edges are $16$, $27$, and one of $03$ or $04$.  By symmetry, we can assume $04$ has low weight.  We can take 
$V(H_{74}) = \{06, 07, 12, 13, 23\}$.

\vspace{-0.2cm}
\subsubsection*{\texorpdfstring{Combinatorial type $H_{75}$}{Combinatorial type H75}}
The only possible multi-multiple edges are $03$, $04$, $16$, and $27$. We can take 
$V(H_{75}) = \{06, 07, 12, 13, 14, 23\}$.

\vspace{-0.2cm}
\subsubsection*{\texorpdfstring{Combinatorial type $H_{76}$}{Combinatorial type H76}}
The only possible multi-multiple edges are $03$, $04$, $27$, and $28$. We can take 
$V(H_{76}) = \{07, 08, 12, 13, 14, 23\}$.

\vspace{-0.2cm}
\subsubsection*{\texorpdfstring{Combinatorial type $H_{77}$}{Combinatorial type H77}}
The only possible multi-multiple edges are one of $03$ or $04$, and one of $26$ or $27$. By symmetry, we can assume $04$ and $26$ have low weight.  We can take 
$V(H_{77}) = \{06, 08, 12, 13, 23\}$.

\vspace{-0.2cm}
\subsubsection*{\texorpdfstring{Combinatorial type $H_{78}$}{Combinatorial type H78}}
The only possible multi-multiple edges are $26$, $27$, and one of $03$ or $04$. By symmetry, we can assume $04$ has low weight.  We can take 
$V(H_{78}) = \{07, 08, 12, 13, 23\}$.

\vspace{-0.2cm}
\subsubsection*{\texorpdfstring{Combinatorial type $H_{79}$}{Combinatorial type H79}}
The only possible multi-multiple edges are $16$, one of $03$ or $04$, and one of $27$ or $28$. By symmetry, we can assume $04$ and $28$ have low weight.  We can take 
$V(H_{79}) = \{02, 06, 07, 12, 13, 23\}$.

\vspace{-0.2cm}
\subsubsection*{\texorpdfstring{Combinatorial type $H_{80}$}{Combinatorial type H80}}
The only possible multi-multiple edges are $02$, $03$, and $24$.  We can take 
$V(H_{80}) = \{04, 08, 12, 13, 23\}$.

\vspace{-0.2cm}
\subsubsection*{\texorpdfstring{Combinatorial type $H_{81}$}{Combinatorial type H81}}
The only possible multi-multiple edges are $16$, $26$, and $28$.  We can take 
$V(H_{81}) = \{02, 06, 08, 26, 68\}$.

\vspace{-0.2cm}
\subsubsection*{\texorpdfstring{Combinatorial type $H_{82}$}{Combinatorial type H82}}
The only possible multi-multiple edges are $02$, $04$, $15$, and $17$.  We can take 
$V(H_{82}) = \{01, 05, 07, 12, 14, 24\}$.

\vspace{-0.2cm}
\subsubsection*{\texorpdfstring{Combinatorial type $H_{83}$}{Combinatorial type H83}}
The only possible multi-multiple edges are $03$, $04$, $15$, and $16$.  We can take 
$V(H_{83}) = \{05, 06, 12, 12, 13, 24\}$.

\vspace{-0.2cm}
\subsubsection*{\texorpdfstring{Combinatorial type $H_{84}$}{Combinatorial type H84}}
The only possible multi-multiple edges are $02$, $04$, $15$, and $16$.  We can take 
$V(H_{84}) = \{01, 04, 06, 12, 14, 24\}$.

\vspace{-0.2cm}
\subsubsection*{\texorpdfstring{Combinatorial type $H_{85}$}{Combinatorial type H85}}
The only possible multi-multiple edges are $02$, $03$, and $13$.  We can take 
$V(H_{85}) = \{01, 08, 12, 13, 23\}$.

\vspace{-0.2cm}
\subsubsection*{\texorpdfstring{Combinatorial type $H_{86}$}{Combinatorial type H86}}
The only possible multi-multiple edges are $03$, $04$, and $13$.  We can take 
$V(H_{86}) = \{08, 12, 13, 14, 23\}$.

\vspace{-0.2cm}
\subsubsection*{\texorpdfstring{Combinatorial type $H_{87}$}{Combinatorial type H87}}
The only possible multi-multiple edges are $04$, $17$, and one of $02$ or $24$. By symmetry, we can assume $24$ has low weight.  We can take 
$V(H_{87}) = \{01, 04, 12, 14, 24\}$.

\vspace{-0.2cm}
\subsubsection*{\texorpdfstring{Combinatorial type $H_{88}$}{Combinatorial type H88}}
The only possible multi-multiple edges are $03$, $04$, $28$, and $34$.  We can take 
$V(H_{88}) = \{04, 08, 12, 14, 23, 24\}$.

\vspace{-0.2cm}
\subsubsection*{\texorpdfstring{Combinatorial type $H_{89}$}{Combinatorial type H89}}
The only possible multi-multiple edges are $03$, $04$, $27$, and $28$.  We can take 
$V(H_{89}) = \{07, 08, 12, 13, 23, 24\}$.

\vspace{-0.2cm}
\subsubsection*{\texorpdfstring{Combinatorial type $H_{90}$}{Combinatorial type H90}}
The only possible multi-multiple edges are $25$, $27$, and one of $03$ or $04$.  By symmetry, we can assume $04$ has low weight.  We can take 
$V(H_{90}) = \{05, 07, 12, 13, 23\}$.

\vspace{-0.2cm}
\subsubsection*{\texorpdfstring{Combinatorial type $H_{91}$}{Combinatorial type H91}}
The only possible multi-multiple edges are $02$, $04$, $16$, and $24$.  We can take 
$V(H_{91}) = \{02, 04, 06, 14, 24, 26\}$.

\vspace{-0.2cm}
\subsubsection*{\texorpdfstring{Combinatorial type $H_{92}$}{Combinatorial type H92}}
The only possible multi-multiple edges are $03$, $04$, $28$, and $34$.  We can take 
$V(H_{92}) = \{03, 08, 12, 14, 23, 24\}$.

\vspace{-0.2cm}
\subsubsection*{\texorpdfstring{Combinatorial type $H_{93}$}{Combinatorial type H93}}
The only possible multi-multiple edges are $03$, $04$, $15$, $16$, and $25$.  We can take 
$V(H_{93}) = \{01, 05, 06, 12, 13, 14, 23\}$.

\vspace{-0.2cm}
\subsubsection*{\texorpdfstring{Combinatorial type $H_{94}$}{Combinatorial type H94}}
The only possible multi-multiple edges are $05$, $13$, and $15$.  We can take 
$V(H_{94}) = \{01, 05, 08, 15, 18, 25\}$.

\vspace{-0.2cm}
\subsubsection*{\texorpdfstring{Combinatorial type $H_{95}$}{Combinatorial type H95}}
The only possible multi-multiple edges are $04$, $08$, $14$, and $28$.  We can take 
$V(H_{95}) = \{01, 04, 08, 14, 24, 28, 48\}$.

\vspace{-0.2cm}
\subsubsection*{\texorpdfstring{Combinatorial type $H_{96}$}{Combinatorial type H96}}
The only possible multi-multiple edges are $02$, $04$, $23$, and $24$.  We can take 
$V(H_{96}) = \{02, 04, 08, 12, 14, 24\}$.

\vspace{-0.2cm}
\subsubsection*{\texorpdfstring{Combinatorial type $H_{97}$}{Combinatorial type H97}}
The only possible multi-multiple edges are $03$, $04$, $13$, $24$, and one of  $16$, $17$, $26$, or $27$.  By symmetry between vertices $6$ and $7$, we can assume $17$, $26$, and $27$ have low weight.  We can take 
$V(H_{97}) = \{01, 02, 06, 12, 14, 23, 24\}$.

\vspace{-0.2cm}
\subsubsection*{\texorpdfstring{Combinatorial type $H_{98}$}{Combinatorial type H98}}
The only possible multi-multiple edge is one of $03$ or $04$; by symmetry, we can assume $04$ has low weight.  We can take 
$V(H_{98}) = \{12, 13, 23\}$.

\vspace{-0.2cm}
\subsubsection*{\texorpdfstring{Combinatorial type $H_{99}$}{Combinatorial type H99}}
The only possible multi-multiple edges are $12$, $16$, $25$, and one of $03$ or $04$.  By symmetry, we can assume $04$ has low weight.  We can take 
$V(H_{99}) = \{01, 02, 05, 06, 12, 13, 16, 25\}$.

\vspace{-0.2cm}
\subsubsection*{\texorpdfstring{Combinatorial type $H_{100}$}{Combinatorial type H100}}
The only possible multi-multiple edges are $03$, $04$, $13$, and $34$.  We can take 
$V(H_{100}) = \{01, 04, 12, 13, 14, 24\}$.

\vspace{-0.2cm}
\subsubsection*{\texorpdfstring{Combinatorial type $H_{101}$}{Combinatorial type H101}}
The only possible multi-multiple edges are one of $03$ or $04$, one of $15$ or $16$, and one of $27$ or $28$. By symmetry, we can assume $04$, $16$, and $28$ have low weight.  We can take 
$V(H_{101}) = \{01, 02, 05, 12, 13, 23\}$.

\vspace{-0.2cm}
\subsubsection*{\texorpdfstring{Combinatorial type $H_{102}$}{Combinatorial type H102}}
Note that there cannot be multi-multiple edges incident to any of $5$, $6$, $7$, or $8$ since these are all contained in a Lann{\'e}r diagram of size $4$ but no L{\'a}nner diagram of size $3$.  Moreover, no two vertices joined by a dashed edge can be incident to a multi-multiple edge, lest the Lann{\'e}r diagram of size $2$ containing these vertices cannot be connected to the Lann{\'e}r diagram $5678$.  Thus there can be at most one multi-multiple edge in any given weighting, and by symmetry we can assume $03$ is the only multi-multiple edge.  We can then take
$V(H_{102}) = \{01, 03, 04, 12, 13, 23\}$.

\vspace{-0.2cm}
\subsubsection*{\texorpdfstring{Combinatorial type $H_{103}$}{Combinatorial type H103}}
The only possible multi-multiple edges are $14$ and one of $02$, $03$, $26$, or $36$.  By symmetry between vertices $2$ and $3$, we can assume $03$, $26$, and $36$ have low weight.  We can take 
$V(H_{103}) = \{01, 06, 14, 24\}$.

\vspace{-0.2cm}
\subsubsection*{\texorpdfstring{Combinatorial type $H_{104}$}{Combinatorial type H104}}
The only possible multi-multiple edges are $06$, $18$, $23$, and one of $02$, $03$, $28$, or $38$.  By symmetry, we can assume $03$, $28$, and $38$ have low weight.  We can take 
$V(H_{104}) = \{03, 06, 08, 12, 16, 26\}$.

\vspace{-0.2cm}
\subsubsection*{\texorpdfstring{Combinatorial type $H_{105}$}{Combinatorial type H105}}
The only possible multi-multiple edges are $13$, $34$, and $46$. We can take 
$V(H_{105}) = \{03, 04, 06, 13, 23\}$.

\vspace{-0.2cm}
\subsubsection*{\texorpdfstring{Combinatorial type $H_{106}$}{Combinatorial type H106}}
The only possible multi-multiple edges are $18$, $26$, $34$, and one of $03$ or $04$. By symmetry, we can assume $04$ has low weight.  We can take 
$V(H_{106}) = \{01, 05, 08, 12, 13, 14, 23\}$.

\vspace{-0.2cm}
\subsubsection*{\texorpdfstring{Combinatorial type $H_{107}$}{Combinatorial type H107}}
The only possible multi-multiple edges are $34$ and $56$. We can take 
$V(H_{107}) = \{04, 06, 13, 26\}$.

\vspace{-0.2cm}
\subsubsection*{\texorpdfstring{Combinatorial type $H_{108}$}{Combinatorial type H108}}
The only possible multi-multiple edges are one of $07$ or $08$, one of $25$ or $26$, and one of $13$ or $14$. By symmetry, we can assume $08$, $26$, and $14$ have low weight.  We can take 
$V(H_{108}) = \{03, 05, 07, 12, 17, 27\}$.

\vspace{-0.2cm}
\subsubsection*{\texorpdfstring{Combinatorial type $H_{109}$}{Combinatorial type H109}}
First, observe that the multi-multiple edges must be contained either in the subdiagram $678$ or in the subdiagram $012345$.  Moreover, there can be at most one multi-multiple edge among $678$, lest the Lann{\'e}r diagrams of size $2$ cannot be connected to the Lann{\'e}r triangle.  No two vertices joined by a dashed edge can both be incident to a multi-multiple edge for this same reason.  Moreover, no vertex in $012345$ can be adjacent to two multi-multiple edges lest these induce a Lann{\'e}r triangle or violate the previous assertion.  By symmetry, we can now assume the multi-multiple edges are limited to $02$ and $67$.  

Suppose that $67$ is multi-multiple.  Then neither vertex of $67$ is connected to any vertex of $012345$ lest these induce a Lann{\'e}r triangle, so for connectivity reasons $8$ must be joined to at least one vertex of each of $01$, $23$ and $45$.  However, vertex $8$ must also be connected to a vertex of $67$.  Thus, vertex $8$ has degree at least $4$ in some elliptic diagram, which is not possible.  Therefore, we need only consider that $02$ is multi-multiple.  In this case, we can take $V(H_{109}) = \{02, 04, 12, 13, 14, 24, 47, 67\}$.

\vspace{-0.2cm}
\subsubsection*{\texorpdfstring{Combinatorial type $H_{110}$}{Combinatorial type H110}}
The only possible multi-multiple edges are $45$, one of $34$ or $45$, one of $24$ or $25$, and any edges of $678$.  By symmetry of $4$ and $5$, we can assume $34$ and $24$ have low weight.  Since the Lann{\'e}r triangle $678$ must be connected to the other Lann{\'e}r triangles, the edges $25$ and $35$ cannot simultaneously be multi-multiple.  Hence, by symmetry we can assume $35$ has low weight.  For the same reason at most one edge of $678$ can be multi-multiple.  So we can conclude, again invoking symmetry of $678$, that the multi-multiple edges are limited to $25$, $45$, and $67$.  We can then take
$V(H_{110}) = \{01, 02, 05, 06, 12, 15, 16, 23, 24, 26, 35\}$.

\vspace{-0.2cm}
\subsubsection*{\texorpdfstring{Combinatorial type $H_{111}$}{Combinatorial type H111}}
The only possible multi-multiple edges are $23$, $45$, one of $02$ or $03$, and one of $14$ or $15$.  By symmetry, we can assume $03$ and $15$ have low weight.  Moreover, both $02$ and $14$ cannot simultaneously be multi-multiple lest $78$ cannot be connected to $01$, so we can assume $14$ is also low weight.  Thus, our set of multi-multiple edges must be among $23$, $45$, and $03$.  We can then set 
$V(H_{111}) = \{03, 04, 05, 07, 12, 34, 37, 47\}$.

\section{\texorpdfstring{Remarks on Compact Coxeter $6$-Polytopes with $10$ Facets}{Remarks on Compact Coxeter 6-Polytopes with 10 Facets}}\label{sec: dim 6 polys}
Throughout this section, we consider compact Coxeter $6$-polytopes with $10$ facets. Using the methods outlined in Section \ref{sec: comb types d+4}, there is a finite algorithm yielding a superset of the combinatorial types of such polytopes.  Though the process was too computationally intense for the classification to be completed in the current project, we were able to list the combinatorial types in the special case where the missing faces have size only $2$ and $5$.
In this case, there are precisely two combinatorial types:
$$I_1 \text{ with missing face list } \{01, 02, 13, 24567, 34567, 89\}$$
and 
$$I_2 \text{ with missing face list } \{01, 02, 13, 24, 34, 56789\}\,.$$

The methods used to analyse these two combinatorial types are similar to those used extensively in Section \ref{sec: d+5 class}. It is straightforward to check using an appropriate choice of $V(I_1)$ and $V(I_2)$ that there are no compact Coxeter polytopes of these combinatorial types.

First, suppose there is a polytope realising $I_1$, and consider its Coxeter diagram.  The only multi-multiple edges are $08$, $09$, $18$, and $19$.  Moreover, note that vertices $8$ and $9$ cannot simultaneously be incident to a multi-multiple edge, lest the Lann{\' e}r diagram $89$ cannot be connected to one of the Lann{\' e}r diagrams of size $5$.  Thus, by symmetry, we can assume the multi-multiple edges are limited to $08$ and $18$.  Then setting 
$$V(I_1) = \{01, 03, 08, 09, 12, 18, 19, 28, 38\}$$
is sufficient to check computationally that there are no polytopes of type $I_1$.  

Now consider a Coxeter diagram of combinatorial type $I_2$.  The only possible multi-multiple edges must be contained within the subdiagram $01234$.  Moreover, any two vertices joined by a dashed edge in this subdiagram cannot both be incident to multi-multiple edges, lest the Lann{\'e}r subdiagram corresponding to these vertices cannot be connected to the Lann{\'e}r diagram of size $5$.  Hence, there is at most one multi-multiple edge contained within $01234$, and by symmetry we can assume this edge is $03$.  One can then show that there are no compact Coxeter polytopes of this combinatorial type by applying Corollary \ref{cor: pass to computation} with
$$V(I_2) = \{01, 02, 03, 04, 12, 13, 14, 23, 34\}\,.$$

The analysis of these combinatorial types as described above yields the following result.

\begin{thm}\label{thm: d+4 6}
There are no compact Coxeter $6$-polytopes with $10$ facets having missing faces of orders only $2$ and $5$.
\end{thm}

There is one known compact Coxeter $6$-polytope with $10$ facets (see Figure \ref{fig: bug poly}), constructed by Bugaenko \cite{Bug}.  

\begin{figure}[ht]
\begin{center}
\includegraphics[width=.3\linewidth]{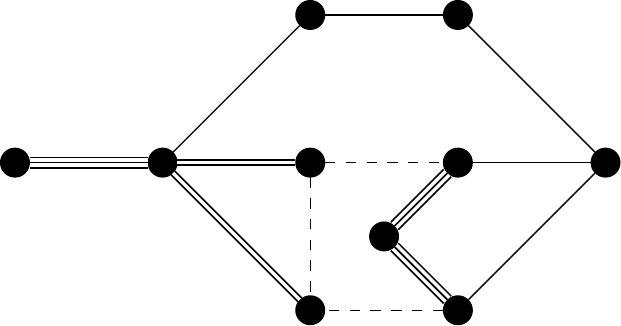}
\caption{The Coxeter diagram of the compact Coxeter $6$-polytope with $10$ facets constructed by Bugaenko.}
\label{fig: bug poly}
\end{center}  
\end{figure}

\addtocontents{toc}{\protect\setcounter{tocdepth}{2}}

\section{\texorpdfstring{Compact $3$-Free Coxeter Polytopes}{Compact 3-Free Coxeter Polytopes}}\label{sec: 3-free}

In this section, we use a modification of Vinberg's proof that simple Coxeter polytopes have dimension at most $29$ to find a tighter bound on the dimension of compact Coxeter $3$-free polytopes, i.e., polytopes where every missing face is of size $2$.  In particular, we show that the dimension of compact $3$-free polytopes is at most $13$. While this result may be of interest in its own right, we make use of this result to bound the dimension of polytopes with few facets in Section \ref{sec: dim bound}.  To our knowledge, the best previous bound on the dimension for this class of polytopes was Vinberg's bound $d \leq 29$ for all compact Coxeter polytopes.

Following the terminology of Esselmann \cite{Ess}, a polytope is called \emph{$k$-free} if it has no missing face of size at least $k$.  

\begin{rem}\label{rem: face k-free}
Every face of a $k$-free polytope is itself a $k$-free polytope (see \cite[Folgerung 1.7]{Ess}).  This holds because given a face $f$ of a polytope $P$, any missing face of $f$ is the intersection of a (possibly larger) missing face of $P$ with $f$. 
\end{rem} 

We first state a lemma from Vinberg's proof of the upper bound on the dimension of a compact polytope which relies on weightings of the planar angles.  A \emph{planar angle} of a polytope is a pair $(A,F)$ where $A$ is a vertex and $F$ is a two-dimensional face containing it; $(A,F)$ is said to be a planar angle \emph{at} $A$ and \emph{on} $F$.

\begin{prop}[\protect{\hspace{1sp}\cite[Prop. 6.2]{Vin}}]\label{prop: planar angle bound}
Let $P$ be a $d$-dimensional compact Coxeter polytope and $c > 0$ a positive number.  We assume that the planar angles of $P$ can be endowed with weights in such a way that:
\begin{enumerate}[(a)]
    \item The sum $\sigma(A)$ of the weights of the planar angles at the vertex $A$ does not exceed $cd$.
    \item The sum $\sigma(F)$ of the weights of the planar angles of any $2$-dimensional face is not less than $5 - k$, where $k$ is the number of vertices of this face.
\end{enumerate}
Then $d < 8c + 6$.
\end{prop}

We also require one of Vinberg's results on the structure of quadrilateral two-dimensional faces.  We first introduce the language of star diagrams; these were referred to as ``star schemes'' by Vinberg.  The \emph{diagram of a planar angle} $(A,F)$ is the Coxeter subdiagram $S_A$ with two chosen ``black'' vertices corresponding to the facets containing $A$ but not $F$.  Note that this is an elliptic diagram of order $d$.  The \emph{star diagram of a face} $F$ is the Coxeter subdiagram $S_F^*$ whose vertices correspond to the facets having a non-trivial intersection with $F$.  We call the vertices of $S_F^*$ that are also vertices of $S_F$, i.e., that correspond to facets containing $F$, ``white'' and the remaining vertices ``black''.  Note that when $F$ is a two-dimensional face, the black vertices correspond to facets whose intersection with $F$ is precisely an edge.

\begin{prop}[\protect{\hspace{1sp}\cite[Prop. 6.4]{Vin}}]\label{prop: quad star scheme}
Let $S_F^*$ be the star diagram of a quadrilateral two-dimensional face $F$ of a polytope $P$. We divide the black vertices of $S_F^*$ into pairs in such a way that the vertices corresponding to opposite sides of $F$ correspond to a single pair. Then
\begin{enumerate}
    \item the removal from $S_F^*$ of any two black vertices in different pairs leaves the diagram of one of the planar angles of $F$;
    \item every hyperbolic subdiagram of $S_F^*$ contains both the black vertices from some pair;
    \item the removal from $S_F^*$ of the two black vertices in the same pair leaves a hyperbolic diagram.
\end{enumerate}
\end{prop}

\begin{thm}\label{thm: 3-free bound}
There are no compact $3$-free Coxeter hyperbolic polytopes of dimension $14$ or higher.
\end{thm}
\begin{proof}
Let $P$ be a compact $3$-free Coxeter $d$-polytope and $S$ its scheme.  We attach weights to the planar angle $P$ similarly to the original construction by Vinberg \cite[Theorem 6.1]{Vin}: the weight of a planar angle $P$ is $1$ if the black vertices in its scheme are adjacent and $0$ otherwise.  We now check that the hypotheses of Proposition \ref{prop: planar angle bound} are satisfied with $c = 1$.

Since every elliptic Coxeter diagram is a forest (when viewed as an unweighted graph), there are at most $d-1$ edges, i.e., pairs of vertices of distance at most $c = 1$.  Thus, Condition (a) of Proposition \ref{prop: planar angle bound} holds.

It remains to check Condition (b) for two-dimensional triangular and quadrilateral faces.  Faces with five or more vertices vacuously satisfy this condition, as the planar angle weights are all non-negative.  

First, we show that $P$ cannot have a two-dimensional triangular face. By Remark \ref{rem: face k-free}, every face of $P$ is $3$-free as well.  Note that a two-dimensional triangular face itself has three one-dimensional facets, which all together have trivial intersection but of which any two intersect non-trivially.  Thus, this triangular face, viewed as a two-dimensional polytope, has a missing face of order $3$.  Hence, it is not $3$-free, contradicting Remark \ref{rem: face k-free}.

Suppose now that $F$ is a two-dimensional quadrilateral face of $P$.  It follows from Proposition \ref{prop: quad star scheme} that every pair of black vertices of $S_F^*$ is contained in some hyperbolic subdiagram not containing the black vertices of the other pair.  Since we have no missing faces of size other than $2$, this implies that each pair of black vertices corresponds to Lann{\'e}r subdiagram of size $2$; call these subdiagrams $L$ and $M$.  Since the subdiagram induced by $L \cup M$ is connected (lest it be parabolic), there must be two black vertices from separate pairs that are adjacent.  These two black vertices lie in the scheme of the corresponding angle of $F$ at a distance of $1$.   Thus, we have shown $\sigma(F) \geq 1$, as desired.  Therefore, by Proposition \ref{prop: planar angle bound}, we have $d < 8(1)+6 = 14$.
\end{proof}

\section{Bounding the Dimension of Polytopes with Few Facets}\label{sec: dim bound}

We now utilise the classification of compact Coxeter $d$-polytopes with few facets to place upper bounds on the dimensions of polytopes with $d+k$ facets for $5 \leq k \leq 10$.  Our methods also utilise the bound $d \leq 14$ for compact $3$-free Coxeter polytopes derived in Section \ref{sec: 3-free}.  

\begin{defn}
Let $D(k)$ denote the maximum positive integer for which a compact Coxeter $D(k)$-polytope with $D(k) + k$ facets exists.
\end{defn}

The only previously published bound on the dimension of such polytopes is the following general result of Vinberg:

\begin{thm}[\protect{\hspace{1sp}\cite[Theorem 4]{Vin2}}]\label{thm: Vinberg 30}
There are no compact Coxeter polytopes in dimension $30$ or higher.
\end{thm}

That is, we have $D(k) \leq 29$ for all $k$.  For $k \leq 4$, the exact value of $D(k)$ is known.

\begin{thm}\label{thm: low k bounds}\hspace{-0.5em}\footnote{Anna Felikson and Pavel Tumarkin have shown in unpublished notes that $D(5) \leq 8$ using a fairly involved argument [Tumarkin, personal communication (2021)].}
We have 
\begin{itemize}
    \item $D(1) = 4$, due to Lann{\'e}r \emph{\cite{Lan}};
    \item $D(2) = 5$, due to Kaplinskaja \emph{\cite{Kap}} and Esselmann \emph{\cite{Ess}};
    \item $D(3) = 8$, due to Esselmann \emph{\cite{Ess}};
    \item $D(4) = 7$, due to Felikson and Tumarkin \emph{\cite{FT}}.
\end{itemize}
\end{thm}

 Thus, we begin our investigation with $k = 5$, and proceed until the bounds obtained by our methods are no stronger than Vinberg's Theorem \ref{thm: Vinberg 30}.  In particular, we derive the following bounds in this section:
 
 \begin{thm}
 We have 
\begin{itemize}
    \item $D(5) \leq 9$,
    \item $D(6) \leq 12$,
    \item $D(7) \leq 15$, 
    \item $D(8) \leq 18$, 
    \item $D(9) \leq 22$, and
    \item $D(10) \leq 26$.
\end{itemize}
 \end{thm}

The argument proceeds by first proving a linear bound on $d$ having slope $4$ (with respect to $k$), then applying the classification of polytopes with fewer facets to slightly improve these bounds in particular cases.  The first bound is obtained by iterating part of the argument used by Felikson and Tumarkin to bound $D(4)$ \cite{FT}. 

\begin{lem}\label{lem: face bound}
If $P$ contains a missing face of size $\ell > 2$, then
$$D(k) \leq \max_{1 \leq i \leq k - 1} D(i) + \ell - 1\,.$$
\end{lem}
\begin{proof}
Let $T$ be a Lann{\'e}r subdiagram of order $\ell$.  Let $S_0$ be a subdiagram of the type 
\[
\begin{cases}
 G_2^{(m)} \text{ for some } m \geq 4 & \text{ if } \ell = 3\;;\\
 H_3 \text{ or } B_3 & \text{ if } \ell = 4\;;\\
 H_4 \text{ or } F_4 & \text{ if } \ell = 5\;.
\end{cases}
\]
Since $S_0$ has at least one bad neighbour (the unique vertex of $T\cut S_0$), by Proposition \ref{prop: Coxeter face} we have that $P(S_0)$ is a Coxeter $(d-\ell + 1)$-polytope with at most $d + k - (\ell + 1)  = (d - \ell + 1) + k - 1$ facets.  Hence,
$$d - \ell + 1 \leq \max_{1 \leq i \leq k - 1} D(i)\,.$$
This implies the desired inequality.
\end{proof}

In order to obtain a general bound from this lemma, we make use of the following result of Esselmann.
\begin{lem}[\protect{\cite[Lemma 6.7]{Ess}}]\label{lem: Ess 2d}
A $3$-free compact Coxeter $d$-polytope has at least $2d$ facets, and in the case of equality, it must be a $d$-cube.
\end{lem}

\begin{thm}\label{thm: partial dim bound}
For $k \geq 5$, we have
$$D(k) \leq  \max\left\{\max_{i < k} D(i) + 4, \min\{k-1,13\}\right\}\,.$$
\end{thm}
\begin{proof}
Fix a Coxeter $d$-polytope $P$ with $d+k$ facets.  If $P$ is not $3$-free, then it contains a missing face of order $3$, $4$, or $5$.  The first bound then follows from Lemma \ref{lem: face bound}.

Now suppose $P$ is $3$-free, so by Lemma \ref{lem: Ess 2d}, $P$ either has at least $2d+1$ facets or is a $d$-cube. Since compact Coxeter $d$-cubes only exist in dimension at most $5$ \cite{JT}, any $3$-free compact Coxeter $d$ polytope with $d+k$ facets satisfies $d \leq k-1$ for $k > 5$.  For $k = 5$, we can have $3$-free compact polytopes in dimension $5$ or less, but this is exceeded by the first bound $\max_{i<5}\{D(i) + 4\} = 12$.  We furthermore have from Theorem \ref{thm: 3-free bound} that such polytopes arise in dimension at most $13$.  Hence for $k > 5$, a $3$-free compact Coxeter $d$-polytope with $d+k$ facets satisfies $d \leq \min\{k-1,13\}$, from which the second bound follows.  
\end{proof}

We now proceed to improving this bound for polytopes with few facets, which involves a more detailed argument depending on the classifications of compact Coxeter polytopes with fewer facets.

\subsection{\texorpdfstring{Polytopes with $d+5$ facets}{Polytopes with d+5 facets}}
Theorem \ref{thm: partial dim bound} implies that there are no compact Coxeter $d$-polytopes with $d+5$ facets for $d \geq 13$.  In this section, we additionally show that no such polytopes arise in dimensions $12$, $11$, and $10$.

For the first reduction, we require some knowledge about compact Coxeter $6$-polytopes with $10$ facets.  In fact, it has been verified that any such polytope has a missing face of size $3$ or $4$ using the exhaustive methods detailed in the first portion of this work (see Theorem \ref{thm: d+4 6}).

\begin{prop}
There is no Coxeter $d$-polytope with $d+5$ facets for $d \geq 10$.
\end{prop}
\begin{proof}
Suppose there is a compact Coxeter polytope $P$ of dimension $d \geq 10$ with at most $d+5$ facets; denote its Coxeter diagram by $\Sigma$.  

First suppose that $P$ has a missing face of order $3$,  and let $S_0$ be a subdiagram of type $G_2^{(m)}$ for $m \geq 4$ contained in this missing face.  The face $P(S_0)$ corresponding to the subdiagram $S_0$ is a Coxeter polytope of dimension $d_1 \geq 8$ with at most $d_1+4$ facets.  This must be the unique $8$-polytope with $11$ facets.  Note that Coxeter diagram of this polytope, $\Sigma_{S_0}$, contains two subdiagrams of type $H_4$.  These subdiagrams are preserved in the subdiagram $\overline{S_0}$, call them $T_1$ and $T_2$.  Moreover, there is a vertex $v \in \Sigma$ connected to $T_1$ and $T_2$ by simple edges, as well as distinct vertices $u_i$ connected to $T_i$ by a simple edge for $i \in \{1,2\}$, as this is true in $\Sigma_{S_0}$ and is preserved under the operations described in Proposition \ref{prop: Coxeter face}.  In particular, $T_1$ has at least two bad neighbours in $\Sigma$.  The face corresponding to the subdiagram $T_1$ in $\Sigma$ is then a compact Coxeter polytope of dimension $d_2 \geq 6$ with at most $d_2 + 3$ facets. In particular, it must have dimension $6$ with $9$ facets or dimension $8$ with $11$.  There are four such polytopes, and in each such polytope every subdiagram of type $H_4$ has two bad neighbours joined by simple edges.  Thus, $T_2$ must have at least $3$ bad neighbours joined by simple edges in $\Sigma$, since there are at least two in each of $\overline{S_0}$ and $\overline{T_1}$, with at most one bad neighbour being common to both.  Thus, the face corresponding to $T_2$ in $\Sigma$ must be a compact Coxeter polytope of dimension $d_3 \geq 6$ with at most $d_3 + 2$ facets, but no such polytope exists by Theorem \ref{thm: low k bounds}.  Therefore, we can assume $P$ has no missing faces of order $3$.

Now suppose that $P$ has a missing face of order $4$.  Let $S_0$ be a subdiagram of type $B_3$ or $H_3$ contained in this missing face.  Then the face corresponding to $S_0$ is a compact Coxeter polytope of dimension $d_1 \geq 7$ with at most $d_1 + 4$ facets.  Hence, $P(S_0)$ is either the unique compact Coxeter $7$-polytope with $11$ facets, or the unique compact Coxeter $8$-polytope with $11$ facets.   If $P(S_0)$ is the $7$-polytope, then $\Sigma_{S_0}$ (and hence $\overline{S_0}$, by Proposition \ref{prop: Coxeter face}) contains a subdiagram $S_1$ of type $H_4$ with three bad neighbours, each connected by a simple edge.  Thus, the face corresponding to $S_1$ in $\Sigma$ is a compact Coxeter polytope of dimension $d_2 \geq 6$ with at most $d_2 + 2$ facets, which does not exist.  Thus, $P(S_0)$ is not the unique compact Coxeter $7$-polytope with $11$ facets.  Now supposing  $P(S_0)$ is the unique compact Coxeter $8$-polytope with $11$ facets, the same reasoning as that used in the previous paragraph shows that $\Sigma$ must contain a subdiagram of type $H_4$ with at least three bad neighbours.  This again yields a face which is a compact Coxeter polytope of dimension $d_3 \geq 6$ with at most $d_3 + 2$ facets, which does not exist.  Therefore, we can assume $P$ has no missing faces of order $4$.

Thus, we can now restrict to considering polytopes with missing faces of order only $2$ and $5$.  By Lemma \ref{lem: Ess 2d}, $P$ must contain a missing face of size $5$.  Let $S_0$ be a subdiagram of type $H_4$ or $F_4$ in this missing face.  Then $P(S_0)$ is a compact Coxeter polytope of dimension $d_1 \geq 6$ with at most $d_1 + 4$ facets.  Moreover, since $\Sigma_{S_0} = \overline{S_0}$ by Proposition \ref{prop: Coxeter face}, then $P(S_0)$ also contains missing faces of order only $2$ and $5$.  In particular, $P(S_0)$ is either a $6$-polytope with $10$ facets or the unique compact Coxeter $8$-polytope with $11$ facets, as the unique compact Coxeter $7$-polytope with $11$ facets contains missing faces of order $4$.  If  $P(S_0)$ is the unique compact Coxeter $8$-polytope with $11$ facets, then we reach a contradiction by the same reasoning as above.  Otherwise, $P(S_0)$ must be a compact Coxeter $6$-polytope with $10$ facets containing missing faces of order only $2$ and $5$, but no such polytopes exist by Theorem \ref{thm: d+4 6}. We can now conclude that no such polytope $P$ arises.
\end{proof}

\begin{cor}
We have $D(5) \leq 9$.
\end{cor}

\begin{lem}\label{lem: 9 d+5}
If there exists a $9$-polytope with $14$ facets, then it contains no missing faces of size $3$ nor any multi-multiple edge.
\end{lem}
\begin{proof}
Suppose that $P$ is a compact $9$-polytope with $14$ facets, and let $\Sigma$ be its Coxeter diagram. Assume $P$ contains a missing face of size $3$ or a multi-multiple edge.  In the first case, let $\Sigma_0$ be a subdiagram of type $G_2^{(m)}$ for $m \geq 4$ contained in the missing face of size $3$.  In the latter case, let $\Sigma_0$ be a subdiagram of type $G_2^{(m)}$ for $m\geq 6$.  Then the face corresponding to $\Sigma_0$ is a $7$-polytope with at most $11$ facets.  It must be the unique $7$-polytope with $11$ facets constructed by Bugaenko.  Moreover, by Proposition \ref{prop: Coxeter face} we can obtain the diagram $\overline{\Sigma_0}$ from $\Sigma_{S_0}$ by possibly replacing some double edges by simple edges, or some dashed edges by ordinary or empty edges.  Thus, there is a diagram of type $H_4$ in $\overline{S_0}$ with at least $3$ bad neighbours, call it $\Sigma_1$.  Moreover, $\Sigma_1$ must have precisely three bad neighbours in $\Sigma$, or else the face corresponding to $\Sigma_1$ is a $5$-polytope with at most $6$ facets, which does not exist.  Then the the face corresponding to $\Sigma_1$ in $\Sigma$ is a $5$-polytope with precisely $7$ facets, furthermore containing the fixed missing face of size $3$ or the multi-multiple edge.  However, no such $5$-polytope exists.
\end{proof}

\begin{thm}\label{thm: 9 d+5}
In a $9$-polytope with $14$ facets, every missing face of size $5$ contains a subdiagram of type $H_4$ or $F_4$ with at least two bad neighbours.
\end{thm}
\begin{proof}
Suppose $P$ is such a polytope containing a missing face of size $5$. Let $\Sigma_0$ be a subdiagram of type $F_4$ or $H_4$.  Then the face corresponding to $\Sigma_0$ is a $5$-polytope $P'$ with at most $9$ facets.  Furthermore, if $\Sigma_0$ does not itself have two bad neighbours, then $P'$ has precisely $9$ facets.  By Lemma \ref{lem: 9 d+5}, $P'$ has no missing faces of size $3$.  By the classification in Section \ref{sec: d+5 class}, all $5$-polytopes with $9$ facets not containing missing faces of order $3$ contain a subdiagram of type $H_4$ with at least $2$ bad neighbours.  
\end{proof}

\subsection{\texorpdfstring{Polytopes with $d+6$ facets}{Polytopes with d+6 facets}}
Now that we have established $D(5) \leq 9$, we obtain from Theorem \ref{thm: partial dim bound} that $D(6) \leq 13$.  Using results from the previous section, we can reduce this bound by one dimension.

\begin{lem}
There are no $13$-polytopes with $19$ facets.
\end{lem}
\begin{proof}
Suppose such a polytope exists, call it $P$.  By Lemma \ref{lem: face bound}, $P$ must have missing faces only of size $2$ or $5$.  Moreover, by Lemma \ref{lem: Ess 2d}, $P$ must contain at least one missing face of size $5$.  Let $\Sigma_0$ be a subdiagram of $T$ of type $H_4$ or $F_4$.  Considering the face corresponding to $\Sigma_0$, we obtain a Coxeter polytope $P'$ of dimension $9$ with at most $14$ facets, and furthermore it contains only missing faces of size $2$ or $5$.  Note that in fact $P'$ must contain precisely $14$ facets, since $D(4) < 9$.  Again applying Lemma \ref{lem: Ess 2d}, $P'$ must contain at least one missing face of size $5$.  Thus, by Theorem \ref{thm: 9 d+5}, $P'$ must contain a subdiagram of type $F_4$ or $H_4$ with at least two bad neighbours.  Looking at the face corresponding to this subdiagram in $\Sigma$ yields a face of $P$ that is a Coxeter $9$-polytope with at most $13$ facets.  No such Coxeter polytopes exists, so we reach a contradiction.
\end{proof}

\begin{cor}\label{cor: d+6 bound}
We have $D(6) \leq 12$.
\end{cor}

We now formulate the following lemma, where we consider faces of increasing codimension corresponding to subdiagrams of type $H_4$ or $F_4$ in polytopes with missing faces of sizes either $2$ or $5$.

\begin{lem}\label{lem: comp iter}
Let $P$ be a compact Coxeter $d$-polytope with $d+k$ facets whose missing faces have size either $2$ or $5$, and $k < d$.  Suppose every subdiagram in $P$ of the type $H_4$ or $F_4$ has precisely one bad neighbour (note that such a diagram must always have at least one bad neighbour).  Then $P$ has a subdiagram $S_\ell$ corresponding to a face $P(S_\ell)$ which is a $(d - 4(\ell+1))$-polytope with $d + k - 5(\ell+1)$ facets, where $\ell$ is the maximum non-negative integer such that 
$$d > k + 3\ell\,,$$
and such that $\Sigma_{S_\ell} = \overline{S_\ell}$.
\end{lem}
\begin{proof}
Fix $m$ such that $d \geq k + 3m$.  We proceed by induction on $m$, showing at each step that $P$ has a subdiagram $S_m$ corresponding to a face $P(S_m)$ which is a $(d-4(m+1))$-polytope with $d+k-5(m+1)$ facets satisfying $\Sigma_{S_m} = \overline{S_m}$.  Note that in particular, $P(S_m)$ contains missing faces of size only $2$ and $5$, as its Coxeter diagram is a subdiagram of the diagram for $P$.  The base case $m = 0$ is trivial, as we can view $P$ as a face of itself and take $S_0$ to be the empty diagram.  

Now suppose that we have fixed a subdiagram $S_{m-1}$ with the necessary properties.  Since we have $d > k + 3m $, then we have
$$2(d-4m) > d + k - 5m\,.$$
Thus, by Lemma \ref{lem: Ess 2d}, a $(d-4m)$-polytope with $d + k - 5m$ facets cannot be $3$-free.   Hence the polytope $P(S_{m-1})$ must have a missing face of size $5$.  The corresponding Lann{\'e}r diagram must contain a subdiagram $T_m$ of type $H_4$ or $F_4$.  Let $S_m = S_{m-1} \cup T_m$.  By our hypotheses, $T_m$ has precisely one bad neighbour, hence $P(S_m)$ is a compact Coxeter $(d-4(m+1))$-polytope with $d+k-5(m+1)$ facets, as desired. 
\end{proof}

\begin{lem}\label{lem: 12 18 bad}
Suppose there exist a compact Coxeter $12$-polytope $P$ with $18$ facets whose missing faces have size either $2$ and $5$.  Then $P$ contains a subdiagram of the type $H_4$, $F_4$ with at least $2$ bad neighbours.
\end{lem}
\begin{proof}
By Lemma \ref{lem: comp iter}, $P$ must have a subdiagram $\Sigma_{P'}$ corresponding to a face $P'$ that is a compact Coxeter $4$-polytope with $8$ facets.  However, it can be checked that each of these contains a subdiagram of type $D_4$ with $4$ bad neighbours or of type $B_3$ with at least $3$ bad neighbours, respectively.  Looking at the faces corresponding to these subdiagrams in $\Sigma$ would yield a compact Coxeter $8$-polytope with at most $10$ facets or a compact Coxeter $9$-polytope with at most $12$ facets, both of which do not exist. 
\end{proof}

\subsection{\texorpdfstring{Polytopes with $d+7$ facets}{Polytopes with d+7 facets}}

\begin{lem}
There is no compact Coxeter $16$-polytope with $23$ facets.
\end{lem}
\begin{proof}
By Lemma \ref{lem: comp iter}, this must have a face which is a $4$-polytope with $8$ facets.  But by the argument in Lemma \ref{lem: 12 18 bad}, these each contain either a subdiagram of type $D_4$ with $4$ bad neighbours or $B_3$ with at least $3$ bad neighbours, respectively.  Looking at the face corresponding to such a subdiagram in $\Sigma$ would yield a compact Coxeter $12$-polytope with at most $15$ facets or a compact Coxeter $13$-polytope with at most $17$ facets, both of which do not exist. 
\end{proof}

\begin{cor}
We have $D(7) \leq 15$.
\end{cor}

\begin{lem}\label{lem: 15 22 bad}
Suppose there exist a compact Coxeter $15$-polytope $P$ with $22$ facets whose missing faces have size either $2$ and $5$.  Then $P$ contains a subdiagram of the type $H_4$ or $F_4$ with at least $2$ bad neighbours.
\end{lem}
\begin{proof}
By Lemma \ref{lem: comp iter}, $P$ must have a subdiagram $\Sigma_{P'}$ corresponding to a face $P'$ that is a compact Coxeter $3$-polytope with $7$ facets.  By \cite[Satz 6.9]{Ess}, there is precisely one possible combinatorial type.  It is fairly straightforward to check that no such polytope can have only angles $\pi/2$ or $\pi/3$ (this can be accomplished with the code described in Section \ref{sec: d+4 properties}).  If this contains a subdiagram of the type $G_2^{(m)}$ for $m \geq 4$, then it must have at least $3$ bad neighbours (as any two vertices not joined by a dashed edge in the diagram are together adjacent to at least $3$ dashed edges).  So there is a face of $P'$ which is a compact Coxeter $13$-polytope with at most $17$ facets, which does not exist.
\end{proof}

\subsection{\texorpdfstring{Polytopes with $d+8$ facets}{Polytopes with d+8 facets}}

\begin{lem}
There is no compact Coxeter $19$-polytope with $27$ facets.
\end{lem}
\begin{proof}
 By Lemma \ref{lem: face bound}, $P$ must have missing faces only of size $2$ or $5$.  By Lemma \ref{lem: comp iter}, this must have a face which is a $15$-polytope with $22$ facets.  By Lemma \ref{lem: 15 22 bad}, this must have a subdiagram of type $H_4$ or $F_4$ with at least two bad neighbours.  But then this yields a face which is a compact Coxeter $15$-polytope with at most $21$ facets, contradicting that $D(i) < 15$ for $i \leq 6$.
\end{proof}

\begin{cor}
We have $D(8) \leq 18$.
\end{cor}

While in the previous cases we were able to reduce the dimension by one via showing that there are certain subdiagrams with at least two bad neighbours, we were not able to obtain such a reduction in this case.  Lemma \ref{lem: comp iter} guarantees that such a polytope has a face which is a compact $2$-polytope with $6$ facets.  However, by the classification in \cite{And} there are many of these for which every subdiagram not of type $A_n$ or $D_5$ has at most one bad neighbour.  

\subsection{\texorpdfstring{Polytopes with $d+9$ or $d+10$ facets}{Polytopes with d+9 or d+10 facets}}

As mentioned previously, in these cases we could not reduce the dimension by one.  

\begin{cor}
We have $D(9) \leq 22$.
\end{cor}
\begin{proof}
This follows directly from Theorem \ref{thm: partial dim bound} and the results proved above that $D(k) \leq 18$ for $k < 9$.
\end{proof}

\begin{cor}
We have $D(10) \leq 26$.
\end{cor}
\begin{proof}
This follows directly from Theorem \ref{thm: partial dim bound} and the results proved above that $D(k) \leq 22$ for $k < 10$.
\end{proof}

For $k \geq 11$, the bounds we obtain from repeated application of Theorem \ref{thm: partial dim bound} are weaker than Vinberg's bound $D(k) \leq 29$ \cite{Vin2}.

\section{Further Directions}

By the results presented in this paper, along with the earlier work of Felikson and Tumarkin \cite{FT}, the only dimension where compact Coxeter $d$-polytopes with $d+4$ facets have not been classified is $d = 6$.  A list of potential combinatorial types can be generated in the same manner as in Section \ref{sec: comb types d+4}, but the large number of point set order types with $10$ points makes this computationally challenging.  With such a list in hand, the analysis of each combinatorial type seems likely to be fairly straightforward, as was the case in dimension $5$.  The author plans to continue this work and is hopeful that the computational difficulties can be handled.

We also provide many new examples of compact Coxeter polytopes in dimension $4$ and $5$.  This includes the first known example in dimension higher than three with a dihedral angle of less than $\frac{\pi}{10}$, as well as the first known example in dimension higher than three with an angle of $\frac{\pi}{7}$.  It may be interesting to further study the properties of many of these polytopes, especially those that are essential.

In Section \ref{sec: dim bound}, we provide improved upper bounds on the dimension of compact Coxeter $d$-polytopes with $d+k$ facets for $5 \leq k \leq 10$.  Thus far, there are no compact Coxeter polytopes known in dimension larger than $8$.  These bounds can be viewed as further evidence that there are perhaps no higher dimensional examples, or that such polytopes may be fairly complicated.  It seems quite likely that the bounds provided can be improved by a constant factor with more detailed analysis.  This may also be the case for the bound on compact Coxeter $3$-free $d$-polytopes obtained in Section \ref{sec: 3-free}, as examples are known only up through dimension $5$ (see, e.g., the cubes classified in \cite{JT}).  Already this has been achieved by Alexandrov \cite{Ale} in showing $d \leq 12$, but perhaps the bound can be improved further.  

\section{Acknowledgements}
The majority of this work was completed as part of the author's thesis \cite{Bur} for the Mathematical Sciences Master of Science by Research program at Durham University.  The author would like to sincerely thank her advisor, Pavel Tumarkin, for his expert support throughout the course of this research and for introducing her to this fascinating topic.  His thorough reading and helpful suggestions have been crucial to the development of this work.  She is also grateful to Anna Felikson, Ruth Kellerhals, and Norbert Peyerimhoff for their helpful comments and suggestions.  The author extends her thanks to Jiming Ma and Fangting Zheng, whose similar work helped her to correct a minors errors in the list of polytopes.  Additional thanks are due to the Marshall Commission for making the author's study in the UK possible.

\newpage

\appendix
\section{List of Combinatorial Types}\label{app: Gale diagrams}
The first table contains the missing face list of the $34$ possible combinatorial types of compact hyperbolic $4$-polytopes with $8$ facets having at least one pair of disjoint facets, as described in Subsection \ref{subsec: point set}.  The second table contains the missing face list of the $111$ possible combinatorial types of compact hyperbolic $5$-polytopes with $9$ facets having at least two pairs of disjoint facets.  Each list of missing faces is presented in its lexicographically least form, with respect to relabellings of the facets. A missing face $f = f_{i_1} \cap \dots \cap f_{i_s}$ is denoted by the string $i_1\dots i_s$ where $i_1 < \dots < i_s$.

\renewcommand{\arraystretch}{1.25}
\begin{table}[H]\label{tab: d+4 comb types}
\begin{longtable}{|l|l|l|}
\hline
\textbf{Type} & \textbf{Missing Face List}                              & \textbf{\begin{tabular}[c]{@{}l@{}}Number of\\ Polytopes\end{tabular}} \\ \hline
$G_1$                                                                                       & 01, 02, 03, 12, 14, 25, 3467, 3567, 4567                       & 130                                                                                         \\ \hline
$G_2$                                                                                       & 01, 02, 03, 12, 14, 2567, 34, 3567, 4567                       & 115                                                                                         \\ \hline
$G_3$                                                                                       & 01, 02, 03, 1245, 1456, 27, 36, 37, 4567                       & 49                                                                                         \\ \hline
$G_4$                                                                                       & 01, 23, 45, 67                                                 & 12                                                                                          \\ \hline
$G_5$                                                                                       & 01, 02, 13, 24, 34, 567                                        & 3                                                                                           \\ \hline
$G_6$                                                                                       & 01, 02, 03, 124, 145, 26, 357, 367, 4567                       & 2                                                                                           \\ \hline
$G_7$                                                                                       & 01, 02, 03, 12, 14, 256, 347, 3567, 4567                       & 2                                                                                           \\ \hline
$G_8$                                                                                       & 01, 02, 03, 1245, 146, 257, 36, 37, 4567                       & 1                                                                                           \\ \hline
$G_9$                                                                                       & 01, 02, 03, 14, 25, 367, 4567                                  & 15                                                                                          \\ \hline
$G_{10}$                                                                                    & 01, 02, 034, 134, 156, 256, 27, 347, 567                       & 1                                                                                           \\ \hline
$G_{11}$                                                                                    & 01, 02, 13, 245, 345, 67                                       & 8                                                                                           \\ \hline
$G_{12}$                                                                                    & 01, 02, 034, 134, 15, 256, 267, 347, 567                       & 4                                                                                           \\ \hline
$G_{13}$                                                                                    & 01, 02, 034, 15, 26, 347, 567                                  & 4                                                                                           \\ \hline
$G_{14}$                                                                                    & 01, 023, 145, 236, 237, 46, 57                                 & 2                                                                                           \\ \hline
$G_{15}$                                                                                    & 01, 02, 03, 12, 145, 267, 345, 367, 4567                       & 0                                                                                           \\ \hline
$G_{16}$                                                                                    & 01, 02, 034, 125, 156, 27, 346, 347, 567                       & 0                                                                                           \\ \hline
$G_{17}$                                                                                    & 01, 02, 03, 124, 145, 267, 35, 367, 4567                       & 0                                                                                           \\ \hline
$G_{18}$                                                                                    & 01, 02, 03, 124, 125, 146, 257, 346, 357, 367, 4567            & 0                                                                                           \\ \hline
$G_{19}$                                                                                    & 01, 02, 03, 124, 135, 145, 246, 267, 357, 367, 4567            & 0                                                                                           \\ \hline
$G_{20}$                                                                                    & 01, 02, 034, 135, 156, 247, 26, 347, 567                       & 0                                                                                           \\ \hline
$G_{21}$                                                                                    & 01, 02, 134, 135, 246, 267, 346, 57                            & 0                                                                                           \\ \hline
$G_{22}$                                                                                    & 01, 02, 03, 124, 1456, 27, 356, 37, 4567                       & 0                                                                                           \\ \hline
$G_{23}$                                                                                    & 01, 02, 034, 134, 15, 25, 267, 3467, 567                       & 0                                                                                           \\ \hline

\end{longtable}
\end{table}

\begin{table}[H]
\begin{longtable}{|l|l|l|}
\hline
\textbf{Type} & \textbf{Missing Face List}                              & \textbf{\begin{tabular}[c]{@{}l@{}}Number of\\ Polytopes\end{tabular}}                                                                                                                     \\ \hline
$G_{24}$                                                                                    & 01, 02, 034, 135, 16, 247, 26, 3457, 567                       & 0    \\ \hline
$G_{26}$                                                                                    & 01, 023, 024, 156, 157, 234, 236, 356, 47, 567                 & 0                                                                                         \\ \hline
$G_{27}$                                                                                    & 01, 02, 034, 134, 135, 156, 247, 256, 267, 347, 567            & 0                                                                                           \\ \hline
$G_{28}$                                                                                    & 01, 02, 034, 135, 136, 157, 246, 247, 257, 346, 567            & 0                                                                                           \\ \hline
$G_{29}$                                                                                    & 01, 02, 034, 125, 136, 156, 247, 257, 346, 347, 567            & 0                                                                                           \\ \hline
$G_{30}$                                                                                    & 01, 023, 024, 135, 156, 237, 247, 357, 46, 567                 & 0                                                                                           \\ \hline
$G_{31}$                                                                                    & 01, 023, 024, 035, 126, 157, 167, 234, 246, 345, 357, 467, 567 & 0                                                                                           \\ \hline
$G_{32}$                                                                                    & 01, 023, 024, 035, 124, 146, 156, 237, 247, 356, 357, 467, 567 & 0                                                                                           \\ \hline
$G_{33}$                                                                                    & 01, 023, 024, 035, 146, 157, 167, 234, 235, 246, 357, 467, 567 & 0                                                                                           \\ \hline
$G_{34}$                                                                                    & 01, 023, 024, 035, 126, 137, 167, 245, 246, 345, 357, 467, 567 & 0                                                                                           \\ \hline
\end{longtable}
\end{table}

\begin{table}[H]
\begin{longtable}{|l|l|l|}
\hline
\textbf{Type} & \textbf{Missing Face List}                              & \textbf{\begin{tabular}[c]{@{}l@{}}Number of\\ Polytopes\end{tabular}} \\ \hline
$H_1$         & 01, 02, 03, 12, 14, 25678, 34, 35678, 45678             & 22                                                                     \\ \hline
$H_2$         & 01, 02, 03, 12456, 14567, 28, 37, 38, 45678             & 18                                                                     \\ \hline
$H_3$         & 01, 02, 03, 12, 14, 256, 3478, 35678, 45678             & 6                                                                      \\ \hline
$H_4$         & 01, 02, 03, 1245, 1456, 27, 368, 378, 45678             & 3                                                                      \\ \hline
$H_5$         & 01, 02, 034, 134, 15, 256, 2678, 3478, 5678             & 1                                                                      \\ \hline
$H_6$         & 01, 02, 03, 12, 14, 25, 34678, 35678, 45678             & 1                                                                      \\ \hline
$H_7$         & 01, 02, 03, 12456, 1457, 267, 268, 38, 45678            & 0                                                                      \\ \hline
$H_8$         & 01, 02, 03, 12456, 1457, 268, 37, 38, 45678             & 0                                                                      \\ \hline
$H_9$         & 01, 02, 03, 12, 14, 2567, 348, 35678, 45678             & 0                                                                      \\ \hline
$H_{10}$      & 01, 02, 03, 124, 14567, 28, 3567, 38, 45678             & 0                                                                      \\ \hline
$H_{11}$      & 01, 02, 03, 1245, 14567, 28, 367, 38, 45678             & 0                                                                      \\ \hline
$H_{12}$      & 01, 02, 034, 134, 15, 25, 2678, 34678, 5678             & 0                                                                      \\ \hline
$H_{13}$      & 01, 02, 03, 12, 145, 267, 3458, 3678, 45678             & 0                                                                      \\ \hline
$H_{14}$      & 01, 02, 03, 12, 145, 2678, 345, 3678, 45678             & 0                                                                      \\ \hline
$H_{15}$      & 01, 02, 03, 124, 145, 2678, 35, 3678, 45678             & 0                                                                      \\ \hline
$H_{16}$      & 01, 02, 034, 135, 16, 2478, 26, 34578, 5678             & 0                                                                      \\ \hline
$H_{17}$      & 01, 02, 034, 1345, 16, 26, 278, 34578, 5678             & 0                                                                      \\ \hline
$H_{18}$      & 01, 02, 03, 124, 125, 146, 2578, 346, 3578, 3678, 45678 & 0                                                                      \\ \hline
$H_{19}$      & 01, 02, 03, 1245, 1346, 1456, 257, 278, 368, 378, 45678 & 0                                                                      \\ \hline
$H_{20}$      & 01, 02, 03, 124, 1256, 147, 2568, 347, 3568, 378, 45678 & 0                                                                      \\ \hline
$H_{21}$      & 01, 02, 03, 1245, 1456, 278, 36, 378, 45678             & 0                                                                      \\ \hline
$H_{22}$      & 01, 02, 03, 1245, 1246, 1457, 268, 357, 368, 378, 45678 & 0                                                                      \\ \hline
$H_{23}$      & 01, 02, 03, 1245, 146, 2578, 36, 378, 45678             & 0                                                                      \\ \hline
$H_{24}$      & 01, 02, 03, 124, 1356, 1456, 247, 278, 3568, 378, 45678 & 0                                                                      \\ \hline
$H_{25}$      & 01, 02, 03, 124, 1456, 27, 3568, 378, 45678             & 0                                                                      \\ \hline
$H_{26}$      & 01, 02, 03, 124, 145, 26, 3578, 3678, 45678             & 0                                                                      \\ \hline
$H_{27}$      & 01, 02, 03, 124, 125, 146, 257, 3468, 3578, 3678, 45678 & 0                                                                      \\ \hline
$H_{28}$      & 01, 02, 03, 124, 135, 145, 246, 2678, 3578, 3678, 45678 & 0                                                                      \\ \hline
$H_{29}$      & 01, 02, 03, 124, 1356, 145, 247, 278, 3568, 3678, 45678 & 0                                                                      \\ \hline
$H_{30}$      & 01, 02, 03, 124, 145, 267, 358, 3678, 45678             & 0                                                                      \\ \hline
$H_{31}$      & 01, 02, 03, 124, 1456, 278, 356, 378, 45678             & 0                                                                      \\ \hline
$H_{32}$      & 01, 02, 03, 124, 1456, 278, 356, 378, 45678             & 0                                                                      \\ \hline
$H_{33}$      & 01, 02, 03, 14, 25, 3678, 45678                         & 0   

\\ \hline
\end{longtable}
\end{table}

\begin{table}[H]
\begin{longtable}{|l|l|l|}
\hline
\textbf{Type} & \textbf{Missing Face List}                              & \textbf{\begin{tabular}[c]{@{}l@{}}Number of\\ Polytopes\end{tabular}}

\\ \hline
$H_{34}$      & 01, 02, 03, 14, 256, 378, 45678                         & 0                                   
 \\ \hline
$H_{35}$      & 01, 02, 034, 15, 256, 267, 3478, 5678                   & 0                                                                      \\ \hline
$H_{36}$      & 01, 02, 034, 1256, 157, 268, 347, 348, 5678             & 0                                                                      \\ \hline
$H_{37}$      & 01, 02, 034, 1345, 156, 267, 278, 348, 5678             & 0                                                                      \\ \hline
$H_{38}$      & 01, 02, 13, 2456, 3456, 78                              & 0                                                                      \\ \hline
$H_{39}$      & 01, 02, 034, 134, 156, 2567, 278, 348, 5678             & 0                                                                      \\ \hline
$H_{40}$      & 01, 02, 034, 125, 156, 278, 346, 3478, 5678             & 0                                                                      \\ \hline
$H_{41}$      & 01, 02, 034, 135, 156, 2478, 267, 348, 5678             & 0                                                                      \\ \hline
$H_{42}$      & 01, 023, 1456, 27, 38, 4567, 4568                       & 0                                                                      \\ \hline
$H_{43}$      & 01, 02, 034, 13, 1567, 248, 2567, 348, 5678             & 0                                                                      \\ \hline
$H_{44}$      & 01, 02, 034, 125, 156, 27, 3468, 3478, 5678             & 0                                                                      \\ \hline
$H_{45}$      & 01, 02, 034, 135, 136, 157, 2468, 2478, 2578, 346, 5678 & 0                                                                      \\ \hline
$H_{46}$      & 01, 02, 034, 135, 16, 2478, 2678, 345, 5678             & 0                                                                      \\ \hline
$H_{47}$      & 01, 02, 034, 134, 135, 156, 2478, 2567, 2678, 348, 5678 & 0                                                                      \\ \hline
$H_{48}$      & 01, 02, 034, 135, 1567, 248, 267, 348, 5678             & 0                                                                      \\ \hline
$H_{49}$      & 01, 02, 0345, 1346, 17, 258, 278, 3456, 678             & 0                                                                      \\ \hline
$H_{50}$      & 01, 02, 034, 134, 1567, 2567, 28, 348, 5678             & 0                                                                      \\ \hline
$H_{51}$      & 01, 02, 0345, 1345, 16, 267, 278, 3458, 678             & 0                                                                      \\ \hline
$H_{52}$      & 01, 02, 0345, 1345, 1346, 167, 258, 267, 278, 3458, 678 & 0                                                                      \\ \hline
$H_{53}$      & 01, 02, 0345, 1346, 137, 178, 2456, 258, 278, 3456, 678 & 0                                                                      \\ \hline
$H_{54}$      & 01, 02, 0345, 1345, 136, 167, 2458, 267, 278, 3458, 678 & 0                                                                      \\ \hline
$H_{55}$      & 01, 02, 034, 1256, 1567, 28, 347, 348, 5678             & 0                                                                      \\ \hline
$H_{56}$      & 01, 02, 0345, 134, 136, 167, 2458, 2578, 267, 3458, 678 & 0                                                                      \\ \hline
$H_{57}$      & 01, 02, 034, 1345, 16, 267, 278, 3458, 5678             & 0                                                                      \\ \hline
$H_{58}$      & 01, 02, 0345, 1346, 1347, 168, 257, 258, 268, 3457, 678 & 0                                                                      \\ \hline
$H_{59}$      & 01, 02, 034, 135, 156, 2478, 26, 3478, 5678             & 0                                                                      \\ \hline
$H_{60}$      & 01, 023, 024, 135, 156, 2467, 2478, 38, 4567, 5678      & 0                                                                      \\ \hline
$H_{61}$      & 01, 02, 034, 125, 126, 157, 268, 3457, 3468, 3478, 5678 & 0                                                                      \\ \hline
$H_{62}$      & 01, 023, 024, 1356, 1567, 238, 248, 3568, 47, 5678      & 0                                                                      \\ \hline
$H_{63}$      & 01, 02, 034, 1256, 137, 1567, 248, 2568, 347, 348, 5678 & 0                                                                      \\ \hline
$H_{64}$      & 01, 02, 0345, 16, 27, 3458, 678                         & 0                                                                      \\ \hline
$H_{65}$      & 01, 02, 0345, 134, 167, 258, 267, 3458, 678             & 0                                                                      \\ \hline
$H_{66}$      & 01, 023, 024, 156, 237, 3567, 48, 5678                  & 0                                         
\\ \hline
\end{longtable}
\end{table}

\begin{table}[H]
\begin{longtable}{|l|l|l|}
\hline
\textbf{Type} & \textbf{Missing Face List}                              & \textbf{\begin{tabular}[c]{@{}l@{}}Number of\\ Polytopes\end{tabular}}

 \\ \hline
$H_{67}$      & 01, 02, 134, 135, 146, 2578, 2678, 357, 468             & 0                                                                      \\ \hline
$H_{68}$      & 01, 02, 034, 134, 15, 2567, 2678, 348, 5678             & 0   

\\ \hline
$H_{69}$      & 01, 02, 034, 134, 135, 156, 2478, 256, 2678, 3478, 5678 & 0                                                                      \\ \hline
$H_{70}$      & 01, 02, 034, 135, 1367, 158, 2467, 248, 258, 3467, 5678 & 0                                                                \\ \hline
$H_{71}$      & 01, 02, 034, 125, 1567, 28, 3467, 348, 5678             & 0                                                                      \\ \hline
$H_{72}$      & 01, 02, 034, 135, 136, 157, 2468, 2478, 257, 3468, 5678 & 0                                                                      \\ \hline
$H_{73}$      & 01, 02, 034, 135, 16, 2478, 267, 3458, 5678             & 0                                                                      \\ \hline
$H_{74}$      & 01, 02, 034, 1345, 156, 167, 2348, 267, 278, 3458, 5678 & 0                                                                      \\ \hline
$H_{75}$      & 01, 02, 034, 1345, 136, 167, 2458, 267, 278, 3458, 5678 & 0                                                                      \\ \hline
$H_{76}$      & 01, 02, 034, 1345, 1356, 167, 248, 267, 278, 3458, 5678 & 0                                                                      \\ \hline
$H_{77}$      & 01, 02, 034, 1345, 1567, 267, 28, 348, 5678             & 0                                                                      \\ \hline
$H_{78}$      & 01, 02, 034, 1345, 1346, 157, 257, 268, 278, 3468, 5678 & 0                                                                      \\ \hline
$H_{79}$      & 01, 02, 034, 1345, 156, 26, 278, 3478, 5678             & 0                                                                      \\ \hline
$H_{80}$      & 01, 023, 024, 1356, 1567, 247, 248, 38, 4567, 5678      & 0                                                                      \\ \hline
$H_{81}$      & 01, 02, 0345, 126, 1347, 167, 258, 268, 3457, 3458, 678 & 0                                                                      \\ \hline
$H_{82}$      & 01, 023, 024, 156, 157, 2348, 2368, 3568, 47, 5678      & 0                                                                      \\ \hline
$H_{83}$      & 01, 02, 034, 125, 136, 156, 2478, 2578, 346, 3478, 5678 & 0                                                                      \\ \hline
$H_{84}$      & 01, 023, 024, 135, 156, 2378, 2478, 3578, 46, 5678      & 0                                                                      \\ \hline
$H_{85}$      & 01, 023, 024, 135, 1567, 2467, 248, 38, 4567, 5678      & 0                                                                      \\ \hline
$H_{86}$      & 01, 02, 034, 135, 136, 1578, 246, 2478, 2578, 346, 5678 & 0                                                                      \\ \hline
$H_{87}$      & 01, 023, 024, 1356, 157, 2368, 248, 3568, 47, 5678      & 0                                                                      \\ \hline
$H_{88}$      & 01, 02, 034, 1256, 1357, 1567, 248, 268, 347, 348, 5678 & 0                                                                      \\ \hline
$H_{89}$      & 01, 02, 034, 1345, 1356, 1567, 248, 267, 278, 348, 5678 & 0                                                                      \\ \hline
$H_{90}$      & 01, 02, 034, 125, 1346, 156, 257, 278, 3468, 3478, 5678 & 0                                                                      \\ \hline
$H_{91}$      & 01, 023, 024, 156, 1578, 234, 2378, 3578, 46, 5678      & 0                                                                      \\ \hline
$H_{92}$      & 01, 02, 034, 134, 1356, 1567, 248, 2567, 278, 348, 5678 & 0                                                                      \\ \hline
$H_{93}$      & 01, 02, 034, 125, 136, 156, 2478, 257, 3468, 3478, 5678 & 0                                                                      \\ \hline
$H_{94}$      & 01, 02, 134, 135, 2467, 2678, 3467, 58                  & 0                                                                      \\ \hline
$H_{95}$      & 01, 02, 134, 1356, 2567, 278, 3567, 48                  & 0                                                                      \\ \hline
$H_{96}$      & 01, 023, 024, 1567, 1568, 234, 237, 3567, 48, 5678      & 0                                                                      \\ \hline
$H_{97}$      & 01, 02, 034, 135, 167, 248, 267, 3458, 5678             & 0                                                                      \\ \hline
$H_{98}$      & 01, 02, 034, 134, 156, 256, 278, 3478, 5678             & 0                                                                      \\ \hline
$H_{99}$      & 01, 02, 034, 15, 26, 3478, 5678                         & 0                                     
\\ \hline
\end{longtable}
\end{table}

\begin{table}[H]
\begin{longtable}{|l|l|l|}
\hline
\textbf{Type} & \textbf{Missing Face List}                              & \textbf{\begin{tabular}[c]{@{}l@{}}Number of\\ Polytopes\end{tabular}}

  \\ \hline
$H_{100}$     & 01, 02, 034, 134, 135, 1567, 248, 2567, 2678, 348, 5678 & 0                                                                      \\ \hline
$H_{101}$     & 01, 02, 034, 12, 156, 278, 3456, 3478, 5678             & 0                                                                      \\ \hline
$H_{102}$     & 01, 02, 13, 24, 34, 5678                                & 0       

\\ \hline
$H_{103}$     & 01, 023, 145, 236, 2378, 46, 578                        & 0                                                                      \\ \hline
$H_{104}$     & 01, 023, 1456, 237, 238, 457, 68                        & 0                                                                      \\ \hline
$H_{105}$     & 01, 02, 134, 135, 246, 2678, 346, 578                   & 0                                                                      \\ \hline

$H_{106}$     & 01, 02, 034, 15, 267, 348, 5678\hspace{4.4cm}                         & 0                                                                      \\ \hline
$H_{107}$     &   01, 02, 134, 156, 278, 347, 568                                      & 0                                                                      \\ \hline
$H_{108}$     & 01, 02, 134, 256, 3456, 78                         & 0                                                                      \\ \hline
$H_{109}$     & 01, 23, 45, 678                                & 0                                                                      \\ \hline
$H_{110}$     & 01, 02, 13, 245, 345, 678                              & 0                                                                      \\ \hline
$H_{111}$     & 01, 023, 145, 236, 456, 78                              & 0                                                                      \\ \hline
\end{longtable}
\end{table}
 \newlength{\originalVOffset}
 \newlength{\originalHOffset}
 \setlength{\originalVOffset}{\voffset}   
 \setlength{\originalHOffset}{\hoffset}

 \setlength{\voffset}{0cm}
 \setlength{\hoffset}{0cm}
\includepdf[pages=-,pagecommand=\thispagestyle{empty}]{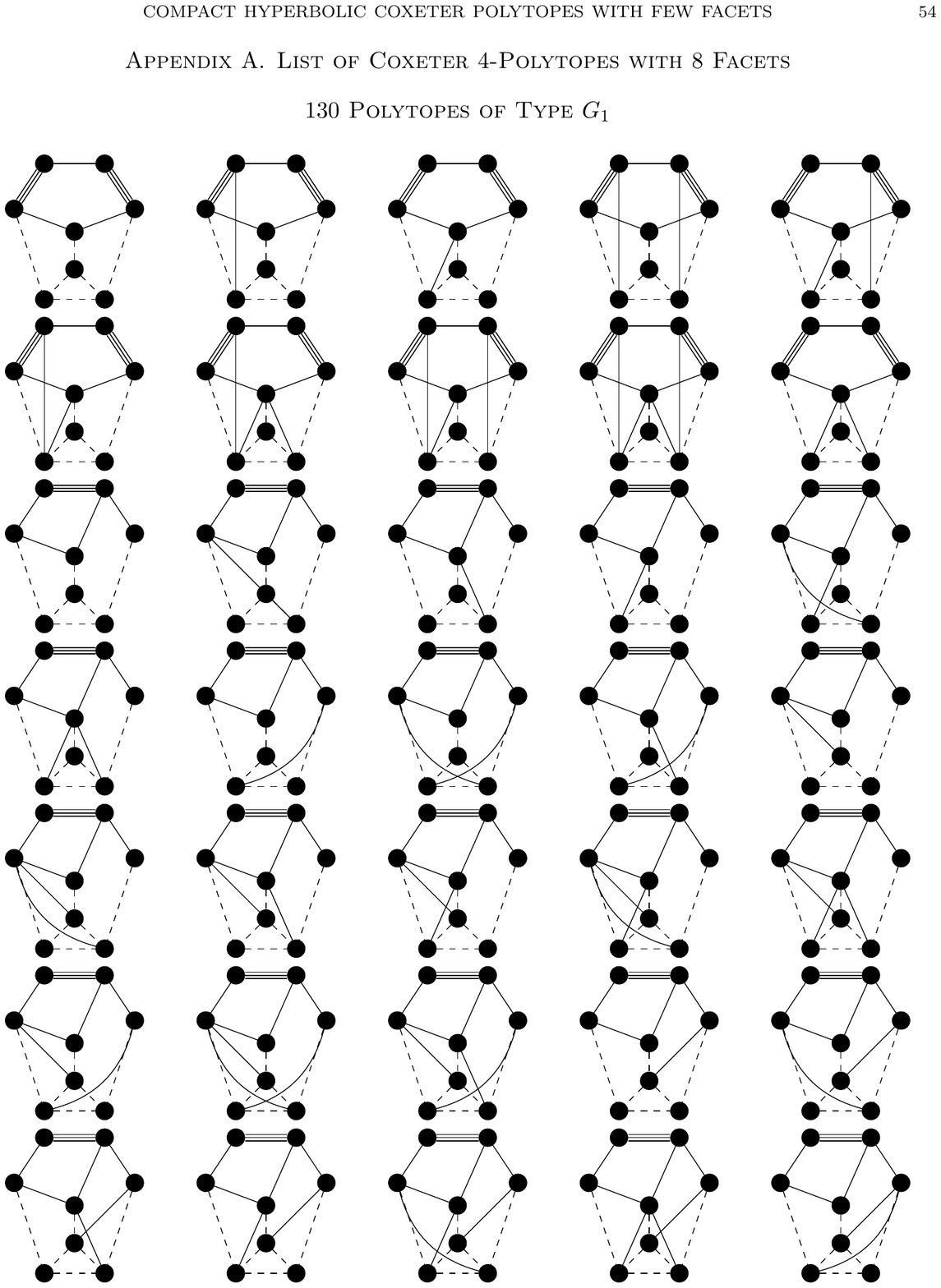}
 \setlength{\voffset}{\originalVOffset}
 \setlength{\hoffset}{\originalHOffset}

\end{document}